\documentclass[pdflatex,sn-mathphys-num]{sn-jnl}

\usepackage{graphicx}
\usepackage{multirow}
\usepackage{booktabs}
\usepackage{amsmath,amssymb,amsfonts,amsthm}
\usepackage{mathrsfs}
\usepackage[title]{appendix}
\usepackage{xcolor}
\usepackage{textcomp}
\usepackage{manyfoot}
\usepackage{algorithm}
\usepackage{algorithmicx}
\usepackage{algpseudocode}
\usepackage{listings}
\usepackage{lipsum}
\usepackage{epstopdf}
\usepackage{tikz}
\usepackage{subcaption}
\usepackage{cleveref}
\usepackage{placeins}

\providecommand{\diag}{\operatorname{diag}}
\providecommand{\Exp}{\operatorname{Exp}}
\providecommand{\grad}{\operatorname{grad}}
\providecommand{\Hess}{\operatorname{Hess}}

\providecommand{\tr}{\operatorname{tr}}
\providecommand{\Sym}{\operatorname{Sym}}
\providecommand{\Skew}{\operatorname{Skew}}

\newcommand{\AP}{\mathrm{AP}}

\providecommand{\SPD}{\Sym_{++}}
\providecommand{\ipF}[2]{\langle #1,#2\rangle_F}
\providecommand{\ipg}[3]{\langle #1,#2\rangle_{#3}}

\theoremstyle{thmstyleone}%
\newtheorem{theorem}{Theorem}
\newtheorem{proposition}[theorem]{Proposition}%

\newtheorem{lemma}[theorem]{Lemma}
\newtheorem{corollary}[theorem]{Corollary}

\theoremstyle{thmstyletwo}%

\theoremstyle{thmstylethree}%
\newtheorem{definition}{Definition}%

\begin{document}

\title[Riemannian Optimization over Symmetric Positive Definite Matrices with the Alpha-Procrustes Geometry]{Riemannian Optimization over Symmetric Positive Definite Matrices with the Alpha-Procrustes Geometry}


\author*[1,3]{\fnm{Derun} \sur{Zhou}}\email{zhouderun@nii.ac.jp}

\author[2,3]{\fnm{Keisuke} \sur{Yano}}\email{yano@ism.ac.jp}

\author[1,3]{\fnm{Mahito} \sur{Sugiyama}}\email{mahito@nii.ac.jp}

\affil[1]{\orgdiv{National Institute of Informatics}, \orgaddress{\street{Hitotsubashi}, \city{Chiyoda-ku}, \postcode{101-8430}, \state{Tokyo}, \country{Japan}}}

\affil[2]{\orgdiv{The Institute of Statistical Mathematics}, \orgaddress{\street{10-3 Midori-cho}, \city{Tachikawa}, \postcode{190-8562}, \state{Tokyo}, \country{Japan}}}

\affil[3]{\orgdiv{The Graduate University for Advanced Studies}, \orgname{SOKENDAI}, \orgaddress{\street{Shonan Village}, \city{Hayama}, \postcode{240-0193}, \state{Kanagawa}, \country{Japan}}}


\abstract{\begin{abstract}

In this paper, we study the Alpha-Procrustes \((\mathrm{AP})\) geometry for
Riemannian optimization on the symmetric positive definite \((\mathrm{SPD})\)
matrix manifold. This geometry forms a one-parameter family that includes
several well-known metrics as special cases, such as the Log-Euclidean
\((\mathrm{LE})\) metric at \(\alpha=0\) and the Bures-Wasserstein
\((\mathrm{BW})\) metric at \(\alpha=1/2\). Our analysis begins with the observation that different choices of \(\alpha\) induce distinct Riemannian metrics, which in turn yield different condition numbers of the Riemannian Hessian at the minimizer.

In particular, we show that, for \(\alpha=1\), the Riemannian Hessian
condition number at a minimizer is bounded by a constant multiple of the
Euclidean Hessian condition number, independently of the condition number of the underlying SPD matrix. This contrasts with the
BW and the Affine-Invariant metrics, whose Riemannian Hessian condition number depends linearly and
quadratically, respectively, on the condition number of the minimizer. More generally, for each fixed \(\alpha\neq 1\), the
\(\mathrm{AP}\)-Riemannian Hessian condition number exhibits a
power-law dependence on the condition number of the underlying
SPD matrix, with exponent \(2|\alpha-1|\).
We further prove that the
\(\mathrm{AP}\) geometry has nonnegative sectional curvature for every
\(\alpha\neq 0\), extending the known curvature property of the \(\mathrm{BW}\) geometry. Combining these results, we show that the metric with \(\alpha=1\) provides stronger local convergence guarantees for Riemannian optimization algorithms in ill-conditioned regimes. We also establish a geodesic convexity
transfer principle: geodesic convexity under the \(\alpha=1/2\)
\((\mathrm{BW})\) geometry can be transferred to general
\(\mathrm{AP}\) geometries through the power transformation
\(P\mapsto P^{2\alpha}\). Extensive numerical experiments on weighted least
squares, trace regression, and the Sylvester equation support our theoretical
findings.
\end{abstract}
}




\maketitle

\section{Introduction}\label{sec1}
Optimization and learning over symmetric positive definite (SPD) matrices play an important role in a wide range of applications, including metric and kernel learning \cite{tsuda2005matrix,guillaumin2009you,suarez2021tutorial}, medical imaging \cite{pennec2006riemannian}, computer vision \cite{harandi2014manifold}, domain adaptation \cite{mahadevan2018unified}, and the modeling of time-varying data \cite{brooks2019exploring}. Recent studies have also explored the use of SPD-valued representations and SPD-aware layers in deep neural networks \cite{huang2017riemannian}.

The set \(\SPD(n)\) forms a smooth manifold \cite{bhatia2009positive}.
Different Riemannian metrics on \(\SPD(n)\) induce different geometric
structures, including Riemannian distances, geodesics, gradients, Hessians,
and curvature. These structures provide the foundation for Riemannian
optimization and geometry-aware learning methods on \(\SPD(n)\)
\cite{absil2008optimization, boumal2023introduction}.

In Riemannian optimization, the choice of metric directly affects the local
behavior of algorithms through the Riemannian Hessian at a minimizer. In
particular, if
\(\kappa:=\kappa\!\bigl(\Hess_g f(P^\star)\bigr)\) denotes the condition
number of the Riemannian Hessian at a local minimizer \(P^\star\), then
\(\kappa\) governs the local convergence rate of first-order methods; for
example, the asymptotic local linear rate of Riemannian gradient descent is
of the form \(1-1/\mathcal{O}(\kappa)\)
\cite{boumal2023introduction}. On the SPD manifold, this condition number
depends on both the Euclidean Hessian of the objective and the chosen
Riemannian metric. Consequently, when the underlying SPD matrix is
ill-conditioned, some commonly used metrics may lead to poorly conditioned
Riemannian Hessians and hence slower local convergence.


A variety of Riemannian metrics have been studied on the \(\mathrm{SPD}\)
manifold, including the Affine-Invariant \((\mathrm{AI})\) metric
\cite{pennec2006riemannian}, the Bures--Wasserstein \((\mathrm{BW})\) metric
\cite{Takatsu2011,malago2018wasserstein,bhatia2019bures}, the
Log-Euclidean \((\mathrm{LE})\) metric \cite{minh2014log}, the Log-Det metric
\cite{sra2012new}, the Log-Cholesky geometry \cite{lin2019riemannian}, and
other constructions motivated by invariance and symmetry
\cite{dryden2009non}. Among these choices, the \(\mathrm{AI}\) and
\(\mathrm{BW}\) metrics are among the most widely used in Riemannian
optimization algorithms \cite{boumal2014manopt}. A key difference is that the
\(\mathrm{BW}\) metric operator depends linearly on the underlying
\(\mathrm{SPD}\) matrix, whereas the \(\mathrm{AI}\) metric operator depends
quadratically. This distinction affects the condition number of the
corresponding Riemannian Hessian and makes the \(\mathrm{BW}\) metric more
robust than the \(\mathrm{AI}\) metric for optimization problems involving
ill-conditioned \(\mathrm{SPD}\) matrices. In addition, many important
optimization problems are geodesically convex under the \(\mathrm{AI}\)
metric, and related geodesic convexity results have also been established
under the \(\mathrm{BW}\) metric \cite{han2021riemannian}. These geometric
structures provide theoretical foundations for efficient Riemannian
optimization on the \(\mathrm{SPD}\) manifold.


In this work, we focus on the Alpha-Procrustes (AP) geometry, a one-parameter family of Riemannian geometries on \(\SPD(n)\) parameterized by \(\alpha\), obtained by generalizing the Procrustes distance optimization problem \cite{minh2022alpha}. As \(\alpha\) varies, this family continuously interpolates between several important SPD geometries, recovering the LE geometry at \(\alpha=0\) and the BW geometry at \(\alpha=1/2\).
Moreover, the main ingredients required for Riemannian optimization under the AP geometry, including the Riemannian gradient and Hessian, can be systematically derived from the associated Riemannian submersion structure. 
These features make the AP geometry a promising framework for optimization over $\SPD(n)$. 
In detail, we investigate the optimization behavior induced by the AP geometry both theoretically and empirically, and show that it provides a viable alternative to the commonly used AI and BW geometries. In particular, our contributions are as follows.

\begin{enumerate}
\item We observe that, when \(\alpha=1\), the eigenvalues of the matrix representation of the Riemannian metric operator remain uniformly bounded, independently of the underlying SPD matrix; this follows from \Cref{thm:riem_hessian_condition_alpha1}. 
As a result, the \(\alpha=1\) metric is better suited than the \(\mathrm{AI}\), \(\mathrm{BW}\), and \(\mathrm{LE}\)
metrics, as well as other fixed \(\mathrm{AP}\) metrics with
\(\alpha\neq 1\), for optimizing ill-conditioned \(\mathrm{SPD}\) matrices, as established by the comparison in \Cref{sec:comparison_ai_ap}.

\item In contrast to the non-positively curved AI geometry, the \(\mathrm{AP}\) geometry, including the \(\mathrm{BW}\) geometry as a special case, has nonnegative sectional curvature, as proved in Proposition~\ref{prop:ap_nonnegative_curvature}. This allows Riemannian optimization algorithms under the \(\mathrm{AP}\) geometry to retain faster convergence rates.

\item For all Riemannian metrics, we analyze the convergence rates of Riemannian steepest descent and Riemannian trust-region methods, and highlight how the condition number of the Riemannian Hessian affects these rates; this connection is formalized in \Cref{thm:rsd_local_convergence,thm:rtr_local_convergence}.

\item We show that geodesic convexity under the \(\mathrm{AP}_{1/2}\) \((\mathrm{BW})\) geometry can be transferred to any \(\mathrm{AP}_{\alpha}\) geometry with \(\alpha\neq 0\): specifically, if \(h\) is geodesically convex under \(\mathrm{AP}_{1/2}\), then the function \(F_{\alpha}(P)=h(P^{2\alpha})\) is geodesically convex under \(\mathrm{AP}_{\alpha}\). This transfer principle is established in
Proposition~\ref{prop:transfer-geodesic-convexity}.

\item We support our analysis with extensive experiments on applications such as weighted least squares, trace regression, and the Sylvester equation. 

\end{enumerate}

The rest of the paper is summarized as follows. In Section~\ref{sec2}, we review the Riemannian metric of the AP Geometry. Section~\ref{sec3} introduces the coordinate representation of the AP Metric. Section~\ref{sec4} develops condition number estimates for the Riemannian Hessian at a minimizer. Section~\ref{sec5} explains how the condition number of the Riemannian Hessian at a minimizer governs the local convergence behavior of different optimization algorithms. 
Section~\ref{sec6} studies geodesic convexity under the AP metric and establishes a transfer principle from the \(\AP_{1/2}\) \((\mathrm{BW})\) geometry to general \(\AP_\alpha\) geometries.
Section~\ref{sec7} presents numerical experiments showing that the metric with \(\alpha=1\) provides a robust geometric framework for Riemannian optimization problems involving ill-conditioned SPD matrices.

\section{Riemannian Metric of the Alpha-Procrustes Geometry}\label{sec2}
In this section, we introduce the Riemannian metric associated with the AP geometry on the manifold of SPD matrices, mainly following \cite{minh2022alpha}. To this end, we first review the manifold structures of the general linear group \(\mathrm{GL}(n)\) and \(\SPD(n)\), together with their tangent spaces.

Let \(\mathrm{M}(n)\) denote the vector space of real \(n \times n\) matrices, equipped with the Frobenius inner product
\begin{equation}
\ipF{X}{Y} = \tr(X^\top Y),
\end{equation}
where \(\tr(\cdot)\) denotes the trace operator.
The general linear group is defined by
\begin{equation}
\mathrm{GL}(n) = \{A \in \mathrm{M}(n) : \det A \neq 0\}.
\end{equation}
Since \(\mathrm{GL}(n)\) is an open subset of \(\mathrm{M}(n)\), it is a smooth manifold of dimension \(n^2\). Accordingly, for any \(A \in \mathrm{GL}(n)\), its tangent space is naturally identified with \(\mathrm{M}(n)\), namely,
\begin{equation}
T_A \mathrm{GL}(n) \cong \mathrm{M}(n).
\end{equation}

The manifold of symmetric positive definite (SPD) matrices is defined by
\begin{equation}
\SPD(n)
=
\left\{
P \in \mathbb{R}^{n\times n}
:\;
P^\top = P,\;
x^\top P x > 0
\quad \text{for all } x \in \mathbb{R}^n \setminus \{0\}
\right\}.
\end{equation}
The set \(\SPD(n)\) is a smooth manifold of dimension \(n(n+1)/2\). Its tangent space at any point \(P \in \SPD(n)\) is naturally identified with \(\Sym(n)\), namely,
\begin{equation}
T_P \SPD(n) \cong \Sym(n),
\end{equation}
where
\begin{equation}
\Sym(n) = \{X \in \mathbb{R}^{n\times n} : X^\top = X\}
\end{equation}
denotes the space of symmetric matrices.

We now introduce the AP metric; see Appendix \ref{sec: details of alpha procrustes geometry} for details.
To this end, we prepare the operator
\(\mathcal L_{P,\alpha} : \Sym(n) \to \Sym(n)\) for each fixed \(P \in \SPD(n)\) and \(\alpha \in \mathbb{R}\) as follows: for any
\(Y \in \Sym(n)\), \(\mathcal L_{P,\alpha}(Y)\) is the unique matrix
\(H \in \Sym(n)\) satisfying
\begin{equation}
\label{eq:LPalphadef}
(D\exp)(\log P)\circ (D\log)(P^{2\alpha})
\big[HP^{2\alpha}+P^{2\alpha}H\big]
=Y.
\end{equation}
In this expression, \(\exp\) denotes the matrix exponential
\[
\exp(X):=\sum_{k=0}^{\infty}\frac{X^k}{k!},
\]
and \(\log\) denotes the principal matrix logarithm, defined for
\(P\in\SPD(n)\) with eigendecomposition
\[
P=Q\,\mathrm{diag}(\lambda_1,\dots,\lambda_n)\,Q^\top,
\]
where \(Q\in \mathrm{O}(n)\) is an orthogonal matrix and \(\lambda_i>0\) are the eigenvalues of \(P\), by
\[
\log(P):=
Q\,\mathrm{diag}(\log\lambda_1,\dots,\log\lambda_n)\,Q^\top.
\]
Moreover, \(D\) denotes the Fr\'echet derivative of a matrix-valued function. More precisely, if
\(f:\mathbb{R}^{n\times n}\to\mathbb{R}^{n\times n}\), then its Fr\'echet derivative at \(A\) in the direction \(E\) is defined by
\[
(Df)(A)[E]
=
\lim_{t\to0}\frac{f(A+tE)-f(A)}{t}.
\]
Using the the operator \(\mathcal L_{P,\alpha}\), the AP metric is defined as
\begin{equation}
\label{eq:riemannian_metric}
g^{(\alpha)}_P(X,Y)
=
4\,\tr\!\Big(
\mathcal L_{P,\alpha}(X)\,P^{2\alpha}\,
\mathcal L_{P,\alpha}(Y)
\Big),
\qquad
P \in \SPD(n),\; X,Y \in \Sym(n).
\end{equation}

\section{Coordinate Representation of the Alpha-Procrustes Metric}\label{sec3}
In this section, we derive the coordinate representation of the Alpha-Procrustes (AP) metric on \(\SPD(n)\). 

\subsection{Coordinate representation of the operator }
Since the Riemannian metric in \eqref{eq:riemannian_metric} involves the operator \(\mathcal L_{P,\alpha}\) defined in \eqref{eq:LPalphadef}, we first need to characterize this operator explicitly. To this end, we begin with recalling the classical definition of a matrix function and the corresponding Dalecki\u{\i}--Kre\u{\i}n formula for its Fr\'echet derivative.
Let \(P\in\SPD(n)\) admit the eigendecomposition
\[
P=Q\Lambda Q^\top,\qquad 
\Lambda=\diag(\lambda_1,\dots,\lambda_n),\ \lambda_i>0,
\]
where \(Q\in O(n)\) is an orthogonal matrix. For a scalar function \(f\) defined on an open interval containing the spectrum
$\sigma(P)=\{\lambda_1,\dots,\lambda_n\}$,
the associated classical matrix function is defined by
\[
f(P):=Q f(\Lambda) Q^\top,
\qquad
f(\Lambda):=\diag\!\bigl(f(\lambda_1),\dots,f(\lambda_n)\bigr).
\]

\begin{lemma}[The Dalecki\u{\i}--Kre\u{\i}n theorem {\normalfont\cite{daletskii1965integration}; see also \cite[Theorem~2.10]{noferini2016dalecki}}]
\label{lem:DaleckiiKrein}
Let \(P\in\SPD(n)\) admit the eigendecomposition $P=Q\Lambda Q^{\top}$.
Assume that \(f\) is continuously differentiable on an open interval containing
\(\sigma(P)\). Then the Fr\'echet derivative of the associated classical matrix function \(f(P)\) at \(P\),
applied to the tangent vector \(E \in \Sym(n)\), is given by
\[
(Df)(P)[E]
=
Q\bigl(F \odot \widetilde E\bigr)Q^\top \quad\text{with}\quad\widetilde E:=Q^\top E Q,
\]
where \(\odot\) denotes the Hadamard product and
\[
F_{ij}
=
\begin{cases}
\dfrac{f(\lambda_i)-f(\lambda_j)}{\lambda_i-\lambda_j}, & \lambda_i\neq\lambda_j,\\[8pt]
f'(\lambda_i), & \lambda_i=\lambda_j.
\end{cases}
\]
\end{lemma}

Using the spectral Fr\'echet derivative characterization in
Lemma~\ref{lem:DaleckiiKrein}, we derive an explicit
entrywise expression for the operator $\mathcal L_{P,\alpha}$ in the eigenbasis of $P$.

\begin{theorem}[Entrywise closed-form of $\mathcal L_{P,\alpha}$ in the eigenbasis]
\label{thm:Lalpha_entrywise}
Let $P\in\SPD(n)$ admit an eigendecomposition $P=Q\Lambda Q^\top$ with
$\Lambda=\mathrm{diag}(\lambda_1,\dots,\lambda_n)$ and $\lambda_i>0$.
Define $\widetilde Y:=Q^\top YQ$ for each $Y\in\Sym(n)$.
Fix $\alpha\in\mathbb{R}\setminus\{0\}$. 
Let $H=\mathcal L_{P,\alpha}(Y)$ be the unique symmetric solution of
\begin{equation}
\label{eq:def_Lalpha_opt}
(D\exp)(\log P)\circ (D\log)(P^{2\alpha})\,[HP^{2\alpha}+P^{2\alpha}H]=Y.
\end{equation}
Then the solution in the eigenbasis of $P$, $\widetilde H=Q^{\top} H Q$, is given by
\begin{equation}
\label{eq:Lalpha_entrywise_opt}
\widetilde H_{ij}=c_{ij}^{(\alpha)}\,\widetilde Y_{ij},
\end{equation}
where
\begin{equation}
\label{eq:cij_alpha_opt}
c_{ii}^{(\alpha)}=\frac{1}{2\lambda_i},\qquad
c_{ij}^{(\alpha)}=
\frac{\lambda_i^{2\alpha}-\lambda_j^{2\alpha}}
{2\alpha\,(\lambda_i-\lambda_j)\,(\lambda_i^{2\alpha}+\lambda_j^{2\alpha})}
\quad (i\neq j).
\end{equation}
Moreover, we have
\[
\lim_{\lambda_j \to \lambda_i}
\frac{\lambda_i^{2\alpha}-\lambda_j^{2\alpha}}
{2\alpha(\lambda_i-\lambda_j)(\lambda_i^{2\alpha}+\lambda_j^{2\alpha})}
=
\frac{1}{4\lambda_i}.
\]
\end{theorem}


\begin{proof}

We start with diagonalizing $P$ as
\[
P=Q\Lambda Q^\top,
\qquad
\Lambda=\mathrm{diag}(\lambda_1,\dots,\lambda_n),\qquad Q\in \mathrm{O}(n),
\]
where $\mathrm{O}(n)$ is the $n\times n$ orthogonal matrices. Along with the decomposition, we set
\[
\widetilde H:=Q^\top H Q,
\qquad
\text{and}
\qquad
\widetilde Y:=Q^\top Y Q.
\]
If $f$ is a classical matrix function,
that is,
the scalar counterpart of \(f\) is continuously differentiable on the open interval $(\min \lambda_{i},\max \lambda_{i})$ with $\min \lambda_{i} >0$,
it satisfies the orthogonal equivariance property
\[
f(QXQ^\top)=Q\,f(X)\,Q^\top
\]
for every orthogonal matrix \(Q\) and every symmetric matrix \(X\) in the domain of \(f\). Consequently, its Fr\'echet derivative obeys the orthogonal equivariance property
\begin{equation}
\label{eq:frechet_orthogonal_equivariance}
(Df)(P)[E]
=
Q\,(Df)(\Lambda)[\widetilde E]\,Q^\top,
\qquad
\widetilde E:=Q^\top E Q.
\end{equation}
Indeed, since
$
P+tE
=
Q(\Lambda+t\widetilde E)Q^\top,
$
we have
\[
f(P+tE)
=
f\!\bigl(Q(\Lambda+t\widetilde E)Q^\top\bigr)
=
Q\,f(\Lambda+t\widetilde E)\,Q^\top,
\]
and hence we obtain
\[
\begin{aligned}
(Df)(P)[E]
&=
\lim_{t\to0}\frac{f(P+tE)-f(P)}{t}\\
&=
Q\left(\lim_{t\to0}\frac{f(\Lambda+t\widetilde E)-f(\Lambda)}{t}\right)Q^\top\\
&=
Q\,(Df)(\Lambda)[\widetilde E]\,Q^\top.
\end{aligned}
\]

We now apply \eqref{eq:frechet_orthogonal_equivariance} to the matrix \(\exp\) and $\log$
to diagonalize \eqref{eq:def_Lalpha_opt}. Since
$
P^{2\alpha}=Q\Lambda^{2\alpha}Q^\top,
$
it follows that
\[
HP^{2\alpha}+P^{2\alpha}H
=
Q\bigl(\widetilde H\Lambda^{2\alpha}+\Lambda^{2\alpha}\widetilde H\bigr)Q^\top.
\]
Applying the orthogonal equivariance property \eqref{eq:frechet_orthogonal_equivariance} of \((D\log)\), we obtain
\[
(D\log)(P^{2\alpha})
\bigl[HP^{2\alpha}+P^{2\alpha}H\bigr]
=
Q\,(D\log)(\Lambda^{2\alpha})
\bigl[\widetilde H\Lambda^{2\alpha}+\Lambda^{2\alpha}\widetilde H\bigr]Q^\top.
\]
Since
$
\log P = Q(\log\Lambda)Q^\top,
$
applying the same property to \((D\exp)\) yields
\[
\begin{aligned}
&(D\exp)(\log P)\circ (D\log)(P^{2\alpha})
\bigl[HP^{2\alpha}+P^{2\alpha}H\bigr]\\
&\qquad=
Q\Bigl((D\exp)(\log \Lambda)\circ (D\log)(\Lambda^{2\alpha})\Bigr)
\bigl[\widetilde H\Lambda^{2\alpha}+\Lambda^{2\alpha}\widetilde H\bigr]Q^\top.
\end{aligned}
\]
Therefore, Equation \eqref{eq:def_Lalpha_opt},
\[
(D\exp)(\log P)\circ (D\log)(P^{2\alpha})
\bigl[HP^{2\alpha}+P^{2\alpha}H\bigr]
=
Y,
\]
is equivalent to
\begin{equation}
\label{eq:def_Lalpha_diag}
(D\exp)(\log \Lambda)\circ (D\log)(\Lambda^{2\alpha})
\bigl[\widetilde H\Lambda^{2\alpha}+\Lambda^{2\alpha}\widetilde H\bigr]
=
\widetilde Y.
\end{equation}

We proceed to get the entrywise expression of $\widetilde Y$.
By Lemma~\ref{lem:DaleckiiKrein}, for a classical matrix function $f$
and the diagonal matrix $\Lambda=\mathrm{diag}(\lambda_1,\dots,\lambda_n)$,
the Fr\'echet derivative satisfies
\begin{equation}
\label{eq:frechet_divdiff}
\big((Df)(\Lambda)[E]\big)_{ij}
=
f^{[1]}(\lambda_i,\lambda_j)\,E_{ij},
\end{equation}
where the divided difference $f^{[1]}$ is defined by
\[
f^{[1]}(a,b)=
\begin{cases}
\dfrac{f(a)-f(b)}{a-b}, & a\neq b,\\[6pt]
f'(a), & a=b.
\end{cases}
\]
Let $S:=\widetilde H\Lambda^{2\alpha}+\Lambda^{2\alpha}\widetilde H$. Since $\Lambda^{2\alpha}$ is diagonal, we have
\begin{equation}
\label{eq:S_entry}
S_{ij}=(\lambda_i^{2\alpha}+\lambda_j^{2\alpha})\,\widetilde H_{ij}.
\end{equation}
Define $T:=D\log(\Lambda^{2\alpha})[S]$. Then by \eqref{eq:frechet_divdiff}, we get the representation
\begin{equation}
\label{eq:T_entry}
T_{ij}=\log^{[1]}(\lambda_i^{2\alpha},\lambda_j^{2\alpha})\,S_{ij}.
\end{equation}
Applying \eqref{eq:frechet_divdiff} to $D\exp(\log\Lambda)$ yields
\begin{equation}
\label{eq:Y_entry_general}
\widetilde Y_{ij}=\exp^{[1]}(\log\lambda_i,\log\lambda_j)\,T_{ij}.
\end{equation}
Combining \eqref{eq:S_entry}-\eqref{eq:Y_entry_general} gives
\begin{equation}
\label{eq:combined_entry}
\widetilde Y_{ij}
=\exp^{[1]}(\log\lambda_i,\log\lambda_j)\,
\log^{[1]}(\lambda_i^{2\alpha},\lambda_j^{2\alpha})\,
(\lambda_i^{2\alpha}+\lambda_j^{2\alpha})\,\widetilde H_{ij}.
\end{equation}

Finally, we express $\exp^{[1]}$ and $\log^{[1]}$ to obtain the closed form of $\widetilde Y$.
If $i=j$, then $\exp'(\log\lambda_i)=\lambda_i$ and $\log'(\lambda_i^{2\alpha})=1/\lambda_i^{2\alpha}$, hence we get
\[\widetilde Y_{ii}=2\lambda_{i}\,\widetilde H_{ii},\] 
which gives $c_{ii}^{(\alpha)}=1/(2\lambda_i)$.
If $i\neq j$, we compute
\[
\exp^{[1]}(\log\lambda_i,\log\lambda_j)=\frac{\lambda_i-\lambda_j}{\log\lambda_i-\log\lambda_j},\qquad
\log^{[1]}(\lambda_i^{2\alpha},\lambda_j^{2\alpha})
=\frac{2\alpha(\log\lambda_i-\log\lambda_j)}{\lambda_i^{2\alpha}-\lambda_j^{2\alpha}},
\]
so their product equals ${2\alpha(\lambda_i-\lambda_j)}/{(\lambda_i^{2\alpha}-\lambda_j^{2\alpha})}$:
\[
\widetilde Y_{ij} =\frac{2\alpha(\lambda_i-\lambda_j)(\lambda_i^{2\alpha}+\lambda_j^{2\alpha})}{(\lambda_i^{2\alpha}-\lambda_j^{2\alpha})} \widetilde H_{ij}.
\]
Solving \eqref{eq:combined_entry} for $\widetilde H_{ij}$ gives \eqref{eq:cij_alpha_opt}.
Note the coincident-eigenvalue limit follows from
\[
\lim_{\lambda_j \to \lambda_i}
\frac{\lambda_i^{2\alpha}-\lambda_j^{2\alpha}}{\lambda_i-\lambda_j}
=
2\alpha \lambda_i^{2\alpha-1},
\qquad
\lim_{\lambda_j \to \lambda_i}
(\lambda_i^{2\alpha}+\lambda_j^{2\alpha})
=
2\lambda_i^{2\alpha},
\]
which completes the proof.
\end{proof}

Moreover, we also derive the entrywise expression of the operator $\mathcal L_{P,\alpha}$ for the Log-Euclidean metric, namely, in the case $\alpha=0$.

\begin{theorem}[Entrywise closed-form of $\mathcal L_{P,0}$ in the eigenbasis (Log-Euclidean case)]
\label{thm:L0_entrywise}
Let $P\in\SPD(n)$ admit an eigendecomposition $P=Q\Lambda Q^\top$ with
$\Lambda=\mathrm{diag}(\lambda_1,\dots,\lambda_n)$ and $\lambda_i>0$.
Define $\widetilde Y:=Q^\top YQ$ for each $Y\in\Sym(n)$.
In the limiting case $\alpha=0$, define $H=\mathcal L_{P,0}(Y)\in\Sym(n)$ by
\begin{equation}
\label{eq:def_L0_opt}
(D\exp)(\log P)\,[\,2H\,]=Y,
\qquad\text{equivalently}\qquad
H=\frac12\,(D\log)(P)[Y].
\end{equation}
Then the solution in the eigenbasis of $P$ is given by
\begin{equation}
\label{eq:L0_entrywise_opt}
\widetilde H_{ij}=c_{ij}^{(0)},
\end{equation}
where
\begin{equation}
\label{eq:cij_0_opt}
c_{ii}^{(0)}=\frac{1}{2\lambda_i},\qquad
c_{ij}^{(0)}=
\frac{\log\lambda_i-\log\lambda_j}{2(\lambda_i-\lambda_j)}
\quad(i\neq j).
\end{equation}
Moreover, we have
\[
\lim_{\lambda_j\to\lambda_i}
\frac{\log\lambda_i-\log\lambda_j}{2(\lambda_i-\lambda_j)}
=
\frac{1}{2\lambda_i}.
\]
\end{theorem}

\begin{proof}

First, consider the inverse form in \eqref{eq:def_L0_opt}.
Since \((D\exp)(\log P)\) is invertible with inverse \((D\log)(P)\), Equation
\eqref{eq:def_L0_opt} implies
\[
H=\frac12\,(D\log)(P)[Y].
\]
Consider the diagonalization of $P$:
\[
P=Q\Lambda Q^\top,
\qquad
\Lambda=\mathrm{diag}(\lambda_1,\dots,\lambda_n),\qquad Q\in O(n),
\]
where $O(n)$ is the $n\times n$ orthogonal matrices. Along with the decomposition, we set
\[
\widetilde H:=Q^\top H Q,
\qquad
\text{and}
\qquad
\widetilde Y:=Q^\top Y Q.
\]

The orthogonal equivariance property \eqref{eq:frechet_orthogonal_equivariance} of the Fr\'{e}chet derivative gives
\[
H=\frac12\,(D\log)(P)[Y]
\quad\Longleftrightarrow\quad
\widetilde H=\frac12\,(D\log)(\Lambda)[\widetilde Y].
\] By Lemma~\ref{lem:DaleckiiKrein}, the Fr\'echet
derivative of $\log$ satisfies
\begin{equation}
\label{eq:Dlog_entry_L0}
\big((D\log)(\Lambda)[\widetilde Y]\big)_{ij}
=
\log^{[1]}(\lambda_i,\lambda_j)\,\widetilde Y_{ij},
\end{equation}
where
\[
\log^{[1]}(a,b)=
\begin{cases}
\dfrac{\log a-\log b}{a-b}, & a\neq b,\\[6pt]
\dfrac{1}{a}, & a=b.
\end{cases}
\]
Therefore, we get
\begin{equation}
\label{eq:H_entry_L0}
\widetilde H_{ij}
=
\frac12\,\log^{[1]}(\lambda_i,\lambda_j)\,\widetilde Y_{ij}.
\end{equation}

We finally consider the explicit representation of $\log^{[1]}$ to obtain the closed form of \eqref{eq:H_entry_L0}.
If $i=j$, then we have
\[
\widetilde H_{ii}
=
\frac{1}{2\lambda_i}\,\widetilde Y_{ii},
\]
implying that $c_{ii}^{(0)}=1/(2\lambda_i)$.
If $i\neq j$, then we have
\[
\widetilde H_{ij}
=
\frac{\log\lambda_i-\log\lambda_j}{2(\lambda_i-\lambda_j)}\,\widetilde Y_{ij},
\]
which gives \eqref{eq:cij_0_opt}. Finally, the coincident-eigenvalue limit is
\[
\lim_{\lambda_j\to\lambda_i}
\frac{\log\lambda_i-\log\lambda_j}{2(\lambda_i-\lambda_j)}
=
\frac12\lim_{\lambda_j\to\lambda_i}
\frac{\log\lambda_i-\log\lambda_j}{\lambda_i-\lambda_j}
=
\frac{1}{2\lambda_i},
\]
which completes the proof.
\end{proof}

\subsection{Coordinate representation of the Riemannian metric}
Based on the Riesz representation theorem, we introduce the metric operator to express the AP metric in terms of the Frobenius inner product.

\begin{definition}[Metric operator $M_\alpha(P)$]
\label{def:Malpha}
Define $M_\alpha(P):\Sym(n)\to\Sym(n)$ as the unique self-adjoint positive-definite linear operator
(with respect to $\ipF{\cdot}{\cdot}$) such that
\begin{equation}
\label{eq:Malpha_def_opt}
g^{(\alpha)}_P(Y,Z)=\ipF{Y}{M_\alpha(P)\,Z}\qquad\text{for all }Y,Z\in\Sym(n).
\end{equation}
\end{definition}

We begin with characterizing the spectral structure of the metric operator \(M_\alpha(P)\) by deriving its eigenvalues with respect to the eigenbasis of \(P\).
Let $\{e_i\}_{i=1}^n$ be the standard basis of $\mathbb R^n$ and define
$E^{(ij)}:=e_ie_j^\top$.
An orthonormal basis of $\Sym(n)$ under the Frobenius inner product
$\ipF{A}{B}:=\tr(A^\top B)$ is given by
\[
\mathbf{E}^{(ii)}:=E^{(ii)},\qquad
\mathbf{E}^{(ij)}:=\frac{1}{\sqrt2}\bigl(E^{(ij)}+E^{(ji)}\bigr)\quad (1\le i<j\le n).
\]
Moreover, for any orthogonal matrix $Q\in\mathbb R^{n\times n}$, the rotated family
\[
\widehat{\mathbf{E}}^{(ii)}:=Q \mathbf{E}^{(ii)}Q^\top,\qquad
\widehat{\mathbf{E}}^{(ij)}:=Q \mathbf{E}^{(ij)}Q^\top\quad(1\le i<j\le n)
\]
remains Frobenius-orthonormal.

\begin{theorem}[Spectrum of the metric operator $M_\alpha(P)$ for $\alpha\neq 0$]
\label{thm:metric_weights_Malpha}
Let $P=Q\Lambda Q^\top\in\SPD(n)$ with $\Lambda=\mathrm{diag}(\lambda_1,\dots,\lambda_n)$ and $\lambda_i>0$.
For $Y,Z\in\Sym(n)$, using $\widetilde Y:=Q^\top YQ$ and $\widetilde Z:=Q^\top ZQ$, we have
\begin{equation}
\label{eq:metric_weights_tilde_opt}
g^{(\alpha)}_P(Y,Z)
=\sum_{i=1}^n w_{ii}^{(\alpha)}\,\widetilde Y_{ii}\widetilde Z_{ii}
+\sum_{1\le i<j\le n} \,w_{ij}^{(\alpha)}\,\widetilde Y_{ij}\widetilde Z_{ij},
\end{equation}
where
\begin{equation}
\label{eq:wii_wij_opt}
w_{ii}^{(\alpha)}=\lambda_i^{2\alpha-2},\qquad
w_{ij}^{(\alpha)}=
\frac{(\lambda_i^{2\alpha}-\lambda_j^{2\alpha})^2}
{\alpha^2(\lambda_i-\lambda_j)^2(\lambda_i^{2\alpha}+\lambda_j^{2\alpha})}
\quad(i\neq j), \quad \alpha \neq 0.
\end{equation}
Equivalently, the Frobenius-orthonormal family
$\{\widehat{\mathbf E}^{(ii)},\widehat{\mathbf E}^{(ij)}\}$
forms an eigenbasis of the metric operator $M_\alpha(P)$:
\[
M_\alpha(P)\widehat{\mathbf E}^{(ii)}
=
w_{ii}^{(\alpha)}\,\widehat{\mathbf E}^{(ii)},
\qquad
M_\alpha(P)\widehat{\mathbf E}^{(ij)}
=
\frac{1}{2}w_{ij}^{(\alpha)}\,\widehat{\mathbf E}^{(ij)}.
\]
\end{theorem}

\begin{proof}

We start with the eigendecomposition $P=Q\Lambda Q^\top$ and define
\[
\widetilde H_Y:=Q^\top \mathcal L_{P,\alpha}(Y)Q,
\qquad
\text{and}\qquad
\widetilde H_Z:=Q^\top \mathcal L_{P,\alpha}(Z)Q.
\]
Since $P^{2\alpha}=Q\Lambda^{2\alpha}Q^\top$ and the trace is invariant under
orthogonal similarity transformations, we obtain
\[
g^{(\alpha)}_P(Y,Z)
=
4\,\tr\!\left(\widetilde H_Y\,\Lambda^{2\alpha}\,\widetilde H_Z\right).
\]

By Theorem~\ref{thm:Lalpha_entrywise}, the entries of $\widetilde H_Y$ and
$\widetilde H_Z$ satisfy
\[
(\widetilde H_Y)_{ij}=c_{ij}^{(\alpha)}\widetilde Y_{ij},
\qquad
(\widetilde H_Z)_{ij}=c_{ij}^{(\alpha)}\widetilde Z_{ij}.
\]
Substituting these expressions yields
\[
\tr(\widetilde H_Y\Lambda^{2\alpha}\widetilde H_Z)
=
\sum^{n}_{i,j=1}
\lambda_j^{2\alpha}
\big(c_{ij}^{(\alpha)}\big)^2
\widetilde Y_{ij}\widetilde Z_{ij}.
\]
Hence we obtain
\[
g^{(\alpha)}_P(Y,Z)
=
4\sum^{n}_{i,j=1}
\lambda_j^{2\alpha}
\big(c_{ij}^{(\alpha)}\big)^2
\widetilde Y_{ij}\widetilde Z_{ij}.
\]

Consider the summand above.
For a diagonal entry $i=j$, using $c_{ii}^{(\alpha)}=1/(2\lambda_i)$ gives
\[
4\lambda_i^{2\alpha}(c_{ii}^{(\alpha)})^2
=
\lambda_i^{2\alpha-2}
=
w_{ii}^{(\alpha)}.
\]
For an off-diagonal entry $i\neq j$, the pair $(i,j)$ and $(j,i)$ together contributes
\[
4(\lambda_i^{2\alpha}+\lambda_j^{2\alpha})
\big(c_{ij}^{(\alpha)}\big)^2
=
\frac{(\lambda_i^{2\alpha}-\lambda_j^{2\alpha})^2}
{\alpha^2(\lambda_i-\lambda_j)^2(\lambda_i^{2\alpha}+\lambda_j^{2\alpha})}
=
w_{ij}^{(\alpha)}.
\]
This, together with the symmetry of $\widetilde Y$ and $\widetilde Z$, establishes \eqref{eq:metric_weights_tilde_opt}-\eqref{eq:wii_wij_opt}.

Since $\{\widehat{\mathbf E}^{(ii)},\widehat{\mathbf E}^{(ij)}\}$
is a Frobenius-orthonormal basis of $\Sym(n)$, comparing
\eqref{eq:Malpha_def_opt} with \eqref{eq:metric_weights_tilde_opt}
yields
\[
M_\alpha(P)\widehat{\mathbf E}^{(ii)}
=
w_{ii}^{(\alpha)}\,\widehat{\mathbf E}^{(ii)},\qquad
M_\alpha(P)\widehat{\mathbf E}^{(ij)}
=
\frac{1}{2}w_{ij}^{(\alpha)}\,\widehat{\mathbf E}^{(ij)},
\]
which concludes the proof.
\end{proof}

We next consider the limiting case $\alpha=0$, which corresponds to the Log-Euclidean metric. The following result gives the corresponding spectral weights of the metric operator $M_0(P)$.

\begin{theorem}[Spectral weights of the metric operator $M_0(P)$ (Log-Euclidean case)]
\label{thm:metric_weights_M0}
Let $P=Q\Lambda Q^\top\in\SPD(n)$ with
$\Lambda=\mathrm{diag}(\lambda_1,\dots,\lambda_n)$ and $\lambda_i>0$.
For $Y,Z\in\Sym(n)$, define
\[
\widetilde Y:=Q^\top YQ,
\qquad
\widetilde Z:=Q^\top ZQ.
\]
In the limiting case $\alpha=0$, the inner product induced by
\[
\ipg{Y}{Z}{P}
:=
4\,\tr\!\left(\mathcal L_{P,0}(Y)\,\mathcal L_{P,0}(Z)\right),
\qquad Y,Z\in\Sym(n),
\]
admits the eigendecomposition
\begin{equation}
\label{eq:metric_weights_tilde_0_opt}
\ipg{Y}{Z}{P}
=
\sum_{i=1}^n w_{ii}^{(0)}\,\widetilde Y_{ii}\widetilde Z_{ii}
+
\sum_{1\le i<j\le n} \,w_{ij}^{(0)}\,\widetilde Y_{ij}\widetilde Z_{ij},
\end{equation}
where
\begin{equation}
\label{eq:wii_wij_0_opt}
w_{ii}^{(0)}=\lambda_i^{-2},
\qquad
w_{ij}^{(0)}
=
2\left(
\frac{\log\lambda_i-\log\lambda_j}{\lambda_i-\lambda_j}
\right)^2
\quad(i\neq j).
\end{equation}
Equivalently, the Frobenius-orthonormal family
$\{\widehat{\mathbf E}^{(ii)},\widehat{\mathbf E}^{(ij)}\}$
forms an eigenbasis of the metric operator $M_0(P)$:
\[
M_0(P)\widehat{\mathbf E}^{(ii)}
=
w_{ii}^{(0)}\,\widehat{\mathbf E}^{(ii)},
\qquad
M_0(P)\widehat{\mathbf E}^{(ij)}
=
\frac{1}{2}w_{ij}^{(0)}\,\widehat{\mathbf E}^{(ij)}.
\]
Moreover, we have
\[
\lim_{\lambda_j\to\lambda_i}
2\left(
\frac{\log\lambda_i-\log\lambda_j}{\lambda_i-\lambda_j}
\right)^2
=
2\lambda_i^{-2}.
\]
\end{theorem}

\begin{proof}
We start with the eigendecomposition $P$ and define
\[
\widetilde H_Y:=Q^\top \mathcal L_{P,0}(Y)Q,
\qquad
\widetilde H_Z:=Q^\top \mathcal L_{P,0}(Z)Q.
\]
Since $P^0=I$ and the trace is invariant under orthogonal similarity
transformations, we obtain
\[
\ipg{Y}{Z}{P}
=
4\,\tr(\widetilde H_Y\,\widetilde H_Z).
\]

By Theorem~\ref{thm:L0_entrywise}, the entries of $\widetilde H_Y$ and
$\widetilde H_Z$ satisfy
\[
(\widetilde H_Y)_{ij}=c_{ij}^{(0)}\widetilde Y_{ij},
\qquad
(\widetilde H_Z)_{ij}=c_{ij}^{(0)}\widetilde Z_{ij},
\]
where
\[
c_{ii}^{(0)}=\frac{1}{2\lambda_i},
\qquad
c_{ij}^{(0)}
=
\frac{\log\lambda_i-\log\lambda_j}{2(\lambda_i-\lambda_j)}
\quad(i\neq j).
\]
Substituting these expressions yields
\[
\tr(\widetilde H_Y\,\widetilde H_Z)
=
\sum_{i,j=1}^n
\bigl(c_{ij}^{(0)}\bigr)^2\,
\widetilde Y_{ij}\widetilde Z_{ij}.
\]

Consider the summand above.
For a diagonal entry $i=j$, using $c_{ii}^{(0)}=1/(2\lambda_i)$ gives
\[
4\bigl(c_{ii}^{(0)}\bigr)^2
=
\lambda_i^{-2}
=
w_{ii}^{(0)}.
\]
For an off-diagonal entry $i\neq j$, the pair $(i,j)$ and $(j,i)$ together contributes
\[
8\bigl(c_{ij}^{(0)}\bigr)^2
=
2\left(
\frac{\log\lambda_i-\log\lambda_j}{\lambda_i-\lambda_j}
\right)^2
=
\,w_{ij}^{(0)}.
\]
This establishes \eqref{eq:metric_weights_tilde_0_opt}--\eqref{eq:wii_wij_0_opt}.

Since $\{\widehat{\mathbf E}^{(ii)},\widehat{\mathbf E}^{(ij)}\}$
is a Frobenius-orthonormal basis of $\Sym(n)$, comparing
Definition~\ref{def:Malpha} with \eqref{eq:metric_weights_tilde_0_opt}
yields
\[
M_0(P)\widehat{\mathbf E}^{(ii)}
=
w_{ii}^{(0)}\,\widehat{\mathbf E}^{(ii)},
\qquad
M_0(P)\widehat{\mathbf E}^{(ij)}
=
\frac{1}{2}w_{ij}^{(0)}\,\widehat{\mathbf E}^{(ij)}.
\]
Finally, we get
\[
\lim_{\lambda_j\to\lambda_i}
\left(
\frac{\log\lambda_i-\log\lambda_j}{\lambda_i-\lambda_j}
\right)^2
=
\left(
\lim_{\lambda_j\to\lambda_i}
\frac{\log\lambda_i-\log\lambda_j}{\lambda_i-\lambda_j}
\right)^2
=
\left(\frac{1}{\lambda_i}\right)^2
=
\lambda_i^{-2},
\]
which completes the proof.
\end{proof}

We now derive the coordinate representation of the metric operator \(M_\alpha(P)\), defined in Definition~\ref{def:Malpha}, with respect to the Frobenius-orthonormal basis \(\{\widehat{\mathbf E}^{(ii)}, \widehat{\mathbf E}^{(ij)}\}\).
Let $P\in\SPD(n)$ and
\begin{equation}
\label{eq:basis_BP}
\mathcal B_P
=
\{\widehat{\mathbf E}^{(ii)}\}_{i=1}^n
\cup
\{\widehat{\mathbf E}^{(ij)}\}_{1\le i<j\le n}
\end{equation}
be the Frobenius-orthonormal basis of $\Sym(n)$ introduced above, satisfying
\[
\ipF{\widehat{\mathbf E}^{(ab)}}{\widehat{\mathbf E}^{(cd)}}=\delta_{(ab),(cd)}.
\]
For any $Y\in\Sym(n)$, its coordinate vector with respect to the basis
$\mathcal B_P$ is defined by
\[
[y]_{\mathcal B_P}
=
\left(
\ipF{Y}{\widehat{\mathbf E}^{(1)}},
\ipF{Y}{\widehat{\mathbf E}^{(2)}},
\dots,
\ipF{Y}{\widehat{\mathbf E}^{(d)}}
\right)^\top
\in\mathbb{R}^{d}.
\]
where $d=\dim\Sym(n)=n(n+1)/2$. Equivalently, we have
\[
Y=\sum_{k=1}^{d} y_k\,\widehat{\mathbf E}^{(k)},
\qquad
y_k=\ipF{Y}{\widehat{\mathbf E}^{(k)}}.
\]
Here we use the single-index notation
\[
\mathcal B_P=\{\widehat{\mathbf E}^{(k)}\}_{k=1}^{d},
\]
obtained by enumerating the elements
$\{\widehat{\mathbf E}^{(ab)}\}_{1\le a\le b\le n}$ in a fixed order.

The matrix representation of the metric operator $M_\alpha(P)$
with respect to the basis $\mathcal B_P$ is defined by
\[ \big([M_\alpha(P)]_{\mathcal B_P}\big)_{k\ell} := g_P^{(\alpha)}\!\left(\widehat{\mathbf E}^{(k)},\,\widehat{\mathbf E}^{(\ell)}\right). \]
By Theorem~\ref{thm:metric_weights_Malpha} and ~\ref{thm:metric_weights_M0}, this matrix is diagonal in the basis $\mathcal B_P$ and takes the form
\begin{equation}
\label{eq:metric_operator_matrix}
[M_{\alpha}(P)]_{\mathcal{B}_{P}}
=
\operatorname{diag}\!\left(
m_{11}^{(\alpha)},\ldots,m_{nn}^{(\alpha)},
m_{12}^{(\alpha)},\ldots,m_{n-1,n}^{(\alpha)}
\right)
\in \mathbb{R}^{d\times d}.
\end{equation}
Here
\[
    m_{ii}^{(\alpha)} := w_{ii}^{(\alpha)},
    \qquad
    m_{ij}^{(\alpha)} := \frac{1}{2}w_{ij}^{(\alpha)},
    \qquad 1\le i<j\le n.
\]
Consequently, for any $Y,Z\in\Sym(n)$,
\begin{equation}
\label{eq:metric_coord_matrix_rewritten}
g^{(\alpha)}_P(Y,Z)
=
[y]_{\mathcal B_P}^{\top}
[M_\alpha(P)]_{\mathcal B_P}
[z]_{\mathcal B_P}.
\end{equation}
Thus, in local coordinates $\mathcal B_P$, the Riemannian metric becomes a weighted
Euclidean inner product on $\mathbb{R}^{d}$, and its matrix representation satisfies
\[
[M_\alpha(P)]_{\mathcal B_P}
\in\mathbb{R}^{d\times d},
\qquad
d=\frac{n(n+1)}2.
\]

\section{Condition number estimates for the Riemannian Hessian at an optimal point}\label{sec4}
In this section, we derive estimates for the condition number of the Riemannian Hessian at an optimal point and use them to compare the robustness of different Riemannian metrics for optimization over ill-conditioned SPD matrices. 

Let \(f : \SPD(n) \to \mathbb{R}\) be twice continuously differentiable, and let \(P^\star \in \SPD(n)\) be a local minimizer of \(f\).
The Euclidean gradient of \(f\) at \(P\) is defined as the unique matrix \(\nabla f(P)\in\Sym(n)\) satisfying
\[
\langle \nabla f(P), Y\rangle_F = (Df)(P)[Y],
\qquad Y\in\Sym(n).
\]
Since \(P^\star\) is a local minimizer, it satisfies
$
\nabla f(P^\star)=0.
$
The Euclidean Hessian of \(f\) at \(P\) is the linear operator
\[
\nabla^2 f(P):\Sym(n)\to\Sym(n),
\qquad
\nabla^2 f(P)[Y]=(D(\nabla f))(P)[Y],
\qquad Y\in\Sym(n).
\]
We assume throughout the paper that \(P^\star\) is nondegenerate, namely,
\[
\langle Y,\nabla^2 f(P^\star)[Y]\rangle_F>0
\qquad
\text{for all } Y\in\Sym(n)\setminus\{0\}.
\]

Next, we derive the Riemannian gradient and Hessian associated with the AP metric defined in \eqref{eq:riemannian_metric}.
The Riemannian gradient of \(f\) at \(P\) with respect to the metric \(g^{(\alpha)}\) is the unique tangent vector \(\grad^{(\alpha)} f(P)\in T_P\SPD(n)\cong\Sym(n)\) satisfying
\[
g^{(\alpha)}_P\!\bigl(\grad^{(\alpha)} f(P),Y\bigr)
=
(Df)(P)[Y],
\qquad
Y\in\Sym(n).
\]
Expressed in the Frobenius-orthonormal basis \(\mathcal B_P\), defined in \eqref{eq:basis_BP}, this relation becomes
\begin{equation}
\label{eq:rgrad_coord}
[\grad^{(\alpha)} f(P)]_{\mathcal B_P}
=
[M_\alpha(P)]_{\mathcal B_P}^{-1}
[\nabla f(P)]_{\mathcal B_P}.
\end{equation}
Now let \(P=Q\Lambda Q^\top\in\SPD(n)\), where
\(\Lambda=\diag(\lambda_1,\dots,\lambda_n)\), and define
\[
\widetilde G_E:=Q^\top \nabla f(P)\, Q .
\]
In the Frobenius-orthonormal basis \(\mathcal B_P\), defined in \eqref{eq:basis_BP}, the coordinates of the Euclidean gradient are given by
\[
([\nabla f(P)]_{\mathcal B_P})_{(ii)}
=
(\widetilde G_E)_{ii},
\qquad
([\nabla f(P)]_{\mathcal B_P})_{(ij)}
=
\sqrt{2}\,(\widetilde G_E)_{ij}.
\]
Since the metric matrix \([M_\alpha(P)]_{\mathcal B_P}\) is diagonal with diagonal entries
\(w_{ii}^{(\alpha)}\) and \((1/2)w_{ij}^{(\alpha)}\), substituting these expressions into \eqref{eq:rgrad_coord} yields
\[
([\grad^{(\alpha)} f(P)]_{\mathcal B_P})_{(ii)}
=
\frac{(\widetilde G_E)_{ii}}{w_{ii}^{(\alpha)}},
\qquad i=1,\dots,n,
\]
and
\[
([\grad^{(\alpha)} f(P)]_{\mathcal B_P})_{(ij)}
=
\frac{2\sqrt{2}\,(\widetilde G_E)_{ij}}{w_{ij}^{(\alpha)}},
\qquad 1\le i<j\le n.
\]
Note that the Riemannian gradient can be derived directly from the horizontal lift induced by the Riemannian submersion structure underlying the AP geometry; please refer to Appendix~\ref{subsec:riemannian_gradient_horizontal_lift}.

Let \(\nabla^{(\alpha)}\) denote the Levi--Civita connection associated with the metric \(g^{(\alpha)}\). The Riemannian Hessian of \(f\) at \(P\) is the linear operator
\[
\Hess^{(\alpha)} f(P):\Sym(n)\to\Sym(n),
\]
defined by
\[
\Hess^{(\alpha)} f(P)[Y]
=
\nabla^{(\alpha)}_Y \grad^{(\alpha)} f(P),
\qquad
Y\in\Sym(n).
\]
In addition, its relation to horizontal lifts under the Riemannian submersion structure of the AP geometry is derived in Appendix~\ref{hession_submersion}.
In particular, at the minimizer \(P^\star\), the identity
\[
\grad^{(\alpha)} f(P^\star)=0
\]
implies that the affine connection term vanishes. Therefore, for any \(Y\in\Sym(n)\),
\[
\Hess^{(\alpha)} f(P^\star)[Y]
=
\bigl(D(\grad^{(\alpha)} f)\bigr)(P^\star)[Y].
\]

Moreover, with respect to the Frobenius-orthonormal basis \(\mathcal B_{P}\), defined in \eqref{eq:basis_BP}, the Euclidean Hessian is represented by the matrix
\begin{equation}
\label{eq:eucl_hessian_matrix}
\bigl([\nabla^{2} f(P)]_{\mathcal B_{P}}\bigr)_{k\ell}
=
\ipF{\widehat{\mathbf E}^{(k)}}{\nabla^{2} f(P)\bigl[\widehat{\mathbf E}^{(\ell)}\bigr]},
\end{equation}
which belongs to \(\mathbb{R}^{d\times d}\), where
$d=\dim\Sym(n)=\frac{n(n+1)}{2}.$
Consequently, with respect to the Frobenius-orthonormal basis
\(\mathcal B_{P^\star}\), defined in \eqref{eq:basis_BP}, the Riemannian
Hessian satisfies
\begin{equation}
\label{eq:riem_hessian_matrix}
[\Hess^{(\alpha)} f(P^\star)]_{\mathcal B_{P^\star}}
=
[M_\alpha(P^\star)]_{\mathcal B_{P^\star}}^{-1}
[\nabla^{2} f(P^\star)]_{\mathcal B_{P^\star}},
\end{equation}
where \([M_\alpha(P^\star)]_{\mathcal B_{P^\star}}\) is defined in
\eqref{eq:metric_operator_matrix}. This identity shows that the metric
operator \(M_\alpha(P^\star)\) acts as a preconditioner for the Euclidean
Hessian.

In many optimization problems, the local convergence behavior of iterative
methods is governed by the condition number of the
Riemannian Hessian at the minimizer
\cite{absil2008optimization, boumal2023introduction}. Motivated by this
observation, we study how the condition number of the Riemannian Hessian depend on
the metric operator. To separate the condition number coming from the objective
function from the condition number induced by the geometry, we introduce the
notation
\[
    H_E(P^\star)
    :=
    [\nabla^{2} f(P^\star)]_{\mathcal B_{P^\star}},
    \qquad
    H_\alpha(P^\star)
    :=
    [\Hess^{(\alpha)} f(P^\star)]_{\mathcal B_{P^\star}}.
\]

Letting \(P\in\SPD\) have the eigendecomposition
\[
    P=Q\Lambda Q^\top,
    \qquad
    \Lambda=\operatorname{diag}(\lambda_1,\ldots,\lambda_n),
    \qquad
    0<\lambda_{\min}(P)\le \lambda_{\max}(P),
\]
we denote the condition number of \(P\) by
\[
    \kappa(P)
    :=
    \frac{\lambda_{\max}(P)}{\lambda_{\min}(P)}.
\]

Moreover, with respect to the eigenbasis \(\mathcal B_P\), 
Theorems~\ref{thm:metric_weights_Malpha} and~\ref{thm:metric_weights_M0} give the diagonal elements of
\([M_\alpha(P)]_{\mathcal B_P}\) in equation \eqref{eq:metric_operator_matrix}. More precisely, for diagonal directions, we have
\[
    m_{ii}^{(\alpha)}=\lambda_i^{2\alpha-2},
    \qquad i=1,\ldots,n,
\]
whereas for off-diagonal directions, we have
\[
    m_{ij}^{(\alpha)}
    =
    \begin{cases}
    \displaystyle
    \frac{1}{2}
    \frac{(\lambda_i^{2\alpha}-\lambda_j^{2\alpha})^2}
    {\alpha^2(\lambda_i-\lambda_j)^2
    (\lambda_i^{2\alpha}+\lambda_j^{2\alpha})},
    & \alpha\neq 0,\\[12pt]
    \displaystyle
    \left(
        \frac{\log\lambda_i-\log\lambda_j}{\lambda_i-\lambda_j}
    \right)^2,
    & \alpha=0,
    \end{cases}
    \qquad 1\le i<j\le n.
\]

The explicit form of the metric operator shows that, when \(\alpha=1\), the
Riemannian Hessian condition number at a minimizer is bounded by a constant
multiple of the Euclidean Hessian condition number, independently of the
condition number of the underlying SPD matrix. This observation is summarized
in the following theorem.

\begin{theorem}[Metric-independent conditioning of the \(\AP_1\) Hessian]
\label{thm:riem_hessian_condition_alpha1}
Let \(f:\SPD\to\mathbb R\) be twice continuously differentiable, and let
\(P^\star\in\SPD\) be a nondegenerate local minimizer of \(f\). Under the
\(\AP_1\) metric, the Riemannian Hessian at \(P^\star\) satisfies
\[
    \lambda_{\min}\!\bigl(H_1(P^\star)\bigr)
    \ge
    \lambda_{\min}\!\bigl(H_E(P^\star)\bigr),
    \qquad
    \lambda_{\max}\!\bigl(H_1(P^\star)\bigr)
    \le
    2\lambda_{\max}\!\bigl(H_E(P^\star)\bigr).
\]
Moreover,
\[
    \lambda_{\max}\!\bigl(H_1(P^\star)\bigr)
    \ge
    \lambda_{\max}\!\bigl(H_E(P^\star)\bigr),
    \qquad
    \lambda_{\min}\!\bigl(H_1(P^\star)\bigr)
    \le
    2\lambda_{\min}\!\bigl(H_E(P^\star)\bigr).
\]
Consequently, we have
\[
    \frac{1}{2}\,\kappa\!\bigl(H_E(P^\star)\bigr)
    \le
    \kappa\!\bigl(H_1(P^\star)\bigr)
    \le
    2\,\kappa\!\bigl(H_E(P^\star)\bigr).
\]
In particular, the conditioning factor induced by the \(\AP_1\) metric is
uniformly bounded by an absolute constant, independently of \(\kappa(P^\star)\).
\end{theorem}

\begin{proof}
We specialize the metric weights to \(\alpha=1\). For diagonal directions,
\[
    m_{ii}^{(1)}=1,
    \qquad i=1,\ldots,n.
\]
For off-diagonal directions, let \(r:=\lambda_j/\lambda_i>0\). Then
\[
\begin{aligned}
    m_{ij}^{(1)}
    &=
    \frac{1}{2}
    \frac{(\lambda_i^2-\lambda_j^2)^2}
    {(\lambda_i-\lambda_j)^2(\lambda_i^2+\lambda_j^2)}        \\
    &=
    \frac{1}{2}
    \frac{(\lambda_i+\lambda_j)^2}{\lambda_i^2+\lambda_j^2}
    =
    \frac{1}{2}
    \frac{(1+r)^2}{1+r^2},
    \qquad 1\le i<j\le n.
\end{aligned}
\]
Since
\[
    1
    \le
    \frac{(1+r)^2}{1+r^2}
    \le
    2,
    \qquad r>0,
\]
we obtain
\[
    \frac{1}{2}
    \le
    m_{ij}^{(1)}
    \le
    1,
    \qquad 1\le i<j\le n.
\]
Together with \(m_{ii}^{(1)}=1\), this implies
\[
    \lambda_{\min}\!\bigl([M_1(P)]_{\mathcal B_P}\bigr)
    \ge
    \frac{1}{2},
    \qquad
    \lambda_{\max}\!\bigl([M_1(P)]_{\mathcal B_P}\bigr)
    =
    1.
\]
Hence
\[
    1
    \le
    \lambda_{\min}\!\bigl([M_1(P)]_{\mathcal B_P}^{-1}\bigr)
    \le
    \lambda_{\max}\!\bigl([M_1(P)]_{\mathcal B_P}^{-1}\bigr)
    \le
    2.
\]

Applying these estimates at \(P=P^\star\), set
\[
    A:=[M_1(P^\star)]_{\mathcal B_{P^\star}}^{-1},
    \qquad
    B:=H_E(P^\star).
\]
Then
\[
    H_1(P^\star)=AB.
\]
Since \(A\) and \(B\) are positive definite, \(AB\) is similar to the
SPD matrix \(A^{1/2}BA^{1/2}\). Hence they have the
same eigenvalues. Therefore,
\[
    \lambda_{\min}(A)\lambda_{\min}(B)
    \le
    \lambda_{\min}\!\bigl(H_1(P^\star)\bigr)
    \le
    \lambda_{\max}(A)\lambda_{\min}(B),
\]
and
\[
    \lambda_{\min}(A)\lambda_{\max}(B)
    \le
    \lambda_{\max}\!\bigl(H_1(P^\star)\bigr)
    \le
    \lambda_{\max}(A)\lambda_{\max}(B).
\]
Using
\[
    1
    \le
    \lambda_{\min}(A)
    \le
    \lambda_{\max}(A)
    \le
    2,
\]
we obtain
\[
    \frac{1}{2}\,\kappa\!\bigl(H_E(P^\star)\bigr)
    \le
    \kappa\!\bigl(H_1(P^\star)\bigr)
    \le
    2\,\kappa\!\bigl(H_E(P^\star)\bigr).
\]
This completes the proof.
\end{proof}

\subsection{Comparison with the AI metric and other \(\AP_\alpha\) metrics}
\label{sec:comparison_ai_ap}
Theorem~\ref{thm:riem_hessian_condition_alpha1} shows that the \(\AP_1\)
metric satisfies the metric-independent bound
\[
    \frac{1}{2}\,\kappa\!\bigl(H_E(P^\star)\bigr)
    \le
    \kappa\!\bigl(H_1(P^\star)\bigr)
    \le
    2\,\kappa\!\bigl(H_E(P^\star)\bigr).
\]
Thus, the local Hessian condition number under the \(\AP_1\) metric depends
only on the Euclidean Hessian condition number
\(\kappa\!\bigl(H_E(P^\star)\bigr)\), and is independent of the condition
number \(\kappa(P^\star)\) of the minimizer.

For comparison, Lemma~1 of \cite{han2021riemannian} gives condition number
bounds for the AI and BW metrics. Specifically, these bounds can be written as
\[
    \frac{\kappa(P^\star)^2}
    {\kappa\!\bigl(H_E(P^\star)\bigr)}
    \le
    \kappa\!\bigl(H_{\mathrm{AI}}(P^\star)\bigr)
    \le
    \kappa(P^\star)^2
    \kappa\!\bigl(H_E(P^\star)\bigr),
\]
and
\[
    \frac{\kappa(P^\star)}
    {\kappa\!\bigl(H_E(P^\star)\bigr)}
    \le
    \kappa\!\bigl(H_{1/2}(P^\star)\bigr)
    \le
    \kappa(P^\star)
    \kappa\!\bigl(H_E(P^\star)\bigr),
\]
where \(H_{1/2}(P^\star)\) corresponds to the \(\AP_{1/2}\) metric, equivalently
\(\mathrm{BW}\) metric. Therefore, the BW metric has a linear dependence on
\(\kappa(P^\star)\), while the AI metric has a quadratic dependence on \(\kappa(P^\star)\). In contrast, Theorem~\ref{thm:riem_hessian_condition_alpha1} shows that the bound for the \(\AP_1\) Hessian condition number is independent of
\(\kappa(P^\star)\).
Compared with the AI metric, the \(\AP_1\) metric is guaranteed to yield a
smaller local Hessian condition number, namely
$
    \kappa\!\bigl(H_1(P^\star)\bigr)
    <
    \kappa\!\bigl(H_{\mathrm{AI}}(P^\star)\bigr),
$
whenever the upper bound for the \(\AP_1\) Hessian condition number is smaller
than the lower bound for the AI Hessian condition number:
\[
    2\,\kappa\!\bigl(H_E(P^\star)\bigr)
    <
    \frac{
        \kappa(P^\star)^2
    }{
        \kappa\!\bigl(H_E(P^\star)\bigr)
    }.
\]
Equivalently, this condition can be written as
\[
    \kappa(P^\star)
    >
    \sqrt{2}\,\kappa\!\bigl(H_E(P^\star)\bigr).
\]

More generally, for any fixed \(\alpha\neq 1\), the same 
argument as in Theorem~\ref{thm:riem_hessian_condition_alpha1} implies that
the lower bound for the \(\AP_\alpha\) Hessian condition number grows at least as
\[
    \kappa\!\bigl(H_\alpha(P^\star)\bigr)
    \ge
    \frac{
        \kappa(P^\star)^{2|\alpha-1|}
    }{
        \kappa\!\bigl(H_E(P^\star)\bigr)
    },
    \qquad \alpha\neq 1.
\]
Therefore, for any fixed \(\alpha\neq 1\), the \(\AP_1\) metric is guaranteed
to yield a smaller local Hessian condition number than the \(\AP_\alpha\)
metric whenever
\[
    \kappa(P^\star)
    >
    \left(
        \sqrt{2}\,\kappa\!\bigl(H_E(P^\star)\bigr)
    \right)^{1/|\alpha-1|}.
\]
Under this condition, we have
\[
    \kappa\!\bigl(H_1(P^\star)\bigr)
    <
    \kappa\!\bigl(H_\alpha(P^\star)\bigr).
\]
Consequently, as the minimizer \(P^\star\) becomes increasingly ill-conditioned,
the \(\AP_1\) geometry provides a more robust local Hessian conditioning bound
than the AI metric and any fixed \(\AP_\alpha\) geometry with
\(\alpha\neq 1\).

\subsection{Example: trace regression objective}
We illustrate the above comparison using a trace regression problem on
\(\SPD\). Let \(A_1,\ldots,A_m\in\Sym(n)\) be fixed sensing matrices and consider
\[
    \min_{P\in\SPD}
    f(P)
    =
    \frac{1}{2m}
    \sum_{i=1}^m
    \left(
        \tr(A_iP)-y_i
    \right)^2,
\]
where \(y=(y_1,\ldots,y_m)\in\mathbb R^m\) is the observed data. In the
rank-one sensing case used in the numerical experiments in Section~\ref{sec7}, we take
\(A_i=a_i a_i^\top\) with \(a_i\in\mathbb R^n\), so that
\[
    \tr(A_iP)=a_i^\top P a_i.
\]
The Euclidean gradient and Hessian are
\[
    \nabla f(P)
    =
    \frac{1}{m}
    \sum_{i=1}^m
    \left(
        \tr(A_iP)-y_i
    \right)A_i,
\]
and
\[
    \nabla^2 f(P)[U]
    =
    \frac{1}{m}
    \sum_{i=1}^m
    \langle A_i,U\rangle_F A_i,
    \qquad U\in\Sym(n).
\]
In particular, if the sensing matrices are chosen so that this Hessian is
positive definite on \(\Sym(n)\), then
\(\kappa\!\bigl(H_E(P^\star)\bigr)\)
is determined only by \(A_1,\ldots,A_m\), and is independent of the
condition number of \(P^\star\).

Now choose the noiseless observations from a prescribed matrix
\(P^\star\in\SPD\), namely
\[
    y_i=\tr(A_iP^\star),
    \qquad i=1,\ldots,m.
\]
Then
\[
    \nabla f(P^\star)=0,
\]
so \(P^\star\) is a minimizer of the trace regression objective. By changing
the observations \(y\), we can therefore change the minimizer \(P^\star\)
without changing the Euclidean Hessian. For example, we may take
\[
    P^\star_\varepsilon
    =
    Q\operatorname{diag}(\varepsilon,1,\ldots,1)Q^\top,
    \qquad \varepsilon>0,
\]
where \(Q\) is orthogonal. Then
\[
    \kappa(P^\star_\varepsilon)
    =
    \frac{1}{\varepsilon}
    \to\infty
    \qquad
    \text{as }\varepsilon\to 0,
\]
while \(\kappa\!\bigl(H_E(P^\star_\varepsilon)\bigr)\) remains fixed.

For the comparison between the \(\AP_1\) and AI metrics, the condition
ensuring that \(\AP_1\) has a smaller Riemannian Hessian condition number is
\[
    \kappa(P^\star)
    >
    \sqrt{2}\,\kappa\!\bigl(H_E(P^\star)\bigr).
\]
Since \(\kappa\!\bigl(H_E(P^\star)\bigr)\) is fixed in this example, this
condition is eventually satisfied as \(\kappa(P^\star)\to\infty\). Therefore,
for sufficiently ill-conditioned minimizers, the \(\AP_1\) metric is
guaranteed to yield a smaller Riemannian Hessian condition number than the AI
metric.

Similarly, for any fixed \(\alpha\neq 1\), the condition ensuring that
\(\AP_1\) has a smaller Riemannian Hessian condition number than
\(\AP_\alpha\) is
\[
    \kappa(P^\star)
    >
    \left(
        \sqrt{2}\,\kappa\!\bigl(H_E(P^\star)\bigr)
    \right)^{1/|\alpha-1|}.
\]
Again, because \(\kappa\!\bigl(H_E(P^\star)\bigr)\) is fixed, this condition
also holds once \(\kappa(P^\star)\) is sufficiently large. Hence, for every
fixed \(\alpha\neq 1\), the estimates imply
\[
    \kappa\!\bigl(H_1(P^\star)\bigr)
    <
    \kappa\!\bigl(H_\alpha(P^\star)\bigr),
    \qquad \alpha\neq 1,
\]
in sufficiently ill-conditioned regimes.

\section{Implications for Local Convergence of Riemannian Optimization Algorithms}\label{sec5}
In this section, we explain how the condition number of the Riemannian Hessian at the minimizer influences the local convergence behavior. As illustrative examples, we consider the Riemannian steepest descent (RSD) method and the Riemannian trust-region (RTR) method. This analysis can also be extended to other Riemannian optimization algorithms.

We first introduce the exponential map formula for \(\alpha\neq 0\). For \(P\in\SPD(n)\) and \(X\in T_P\SPD(n)\simeq\Sym(n)\), the exponential map \cite{lee2018introduction} is defined by
\[
\Exp_P^{(\alpha)}(tX):=\gamma_{P,X}^{(\alpha)}(t),
\]
where \(\gamma_{P,X}^{(\alpha)}\) denotes the geodesic satisfying
\[
\gamma_{P,X}^{(\alpha)}(0)=P,
\qquad
\dot{\gamma}_{P,X}^{(\alpha)}(0)=X,\quad
\text{and}\quad
\nabla^{(\alpha)}_{\dot{\gamma}_{P,X}^{(\alpha)}(t)}
\dot{\gamma}_{P,X}^{(\alpha)}(t)=0.
\]
Moreover, for each \(X\in T_P\SPD(n)\simeq\Sym(n)\), there exists a unique matrix \(Y\in\Sym(n)\) such that
\[
Y=\mathcal L_{P,\alpha}(2\alpha X),\quad
\text{and}
\quad
\Exp_P^{(\alpha)}(X)
=
\Bigl((I+Y)\,P^{2\alpha}\,(I+Y)\Bigr)^{\frac{1}{2\alpha}}.
\]
In the limiting case \(\alpha=0\), corresponding to the Log-Euclidean metric, the associated exponential map is given in \cite{arsigny2007geometric}.
\[
\Exp_P^{(0)}(X)
=
\exp\bigl(\log(P)+(D\log)(P)[X]\bigr).
\]
A proof of the case \(\alpha\neq 0\) is deferred to Theorem~\ref{thm:Exp_alpha_general} in the appendix.

After introducing the exponential map, we now recall the Riemannian distance induced by the \(\AP\) metric.
Let \(d^{(\alpha)}:\SPD(n)\times\SPD(n)\to\mathbb R_{\ge 0}\) denote the
corresponding Riemannian distance, defined by
\[
d^{(\alpha)}(P,Q)
:=
\inf_{\gamma}
\int_0^1
\sqrt{
g^{(\alpha)}_{\gamma(t)}
\big(\dot\gamma(t),\dot\gamma(t)\big)
}\,dt,
\]
where the infimum is taken over all piecewise smooth curves
\(\gamma:[0,1]\to\SPD(n)\) such that \(\gamma(0)=P\) and \(\gamma(1)=Q\).
Here \(\dot\gamma(t)\in T_{\gamma(t)}\SPD(n)\) denotes the tangent vector
field along \(\gamma\).
For \(\alpha\neq 0\), this distance admits the closed-form expression \cite{minh2022alpha}
\begin{equation}
\label{eq:distance_alpha_nonzero}
d^{(\alpha)}(P,Q)
=
\frac{1}{|\alpha|}
\Bigl(
\tr\!\bigl(
P^{2\alpha}+Q^{2\alpha}
-
2\,(P^\alpha Q^{2\alpha}P^\alpha)^{1/2}
\bigr)
\Bigr)^{1/2}.
\end{equation}
In the limiting case \(\alpha=0\), one has
\begin{equation}
\label{eq:distance_alpha_zero}
d^{(0)}(P,Q)
=
\lim_{\alpha\to 0} d^{(\alpha)}(P,Q)
=
\|\log P-\log Q\|_F.
\end{equation}

After introducing the induced distance, we next recall the definition of
sectional curvature, which will be used to characterize the curvature
structure of the \(\AP\) geometry.

\begin{definition}[Sectional curvature {\cite[Chapter~3]{do1992riemannian}}]
Let \((\mathcal M,g)\) be a Riemannian manifold. For a point
\(p\in\mathcal M\) and a two-dimensional tangent plane
\(\sigma\subset T_p\mathcal M\), choose two linearly independent vectors
\(u,v\in T_p\mathcal M\) such that \(\sigma=\operatorname{span}\{u,v\}\).
The sectional curvature of \(\sigma\) at \(p\) is defined by
\[
    K_g(p;\sigma)
    :=
    \frac{
        g_p\bigl(R_p(u,v)v,u\bigr)
    }{
        g_p(u,u)g_p(v,v)-g_p(u,v)^2
    },
\]
where \(R\) is the Riemannian curvature tensor associated with \(g\).
This quantity is independent of the choice of basis \(u,v\) for the plane
\(\sigma\).
\end{definition}

We now apply this notion to the AP geometry and show that its sectional curvature is always nonnegative.

\begin{proposition}[Nonnegative sectional curvature of \(\AP\)]
\label{prop:ap_nonnegative_curvature}
For every \(\alpha\neq 0\), the \(\AP\) geometry on \(\SPD\) has
nonnegative sectional curvature.
\end{proposition}

The proof is provided in Appendix~\ref{sec:proof_ap_nonnegative_curvature}.

The nonnegative curvature property will be used to simplify the curvature
factor in the local convergence rate for Riemannian steepest descent.
We first recall the corresponding local convergence result from
\cite[Theorem~1]{han2021riemannian}.

\begin{theorem}[Local convergence of Riemannian steepest descent {\cite[Theorem~1]{han2021riemannian}}]
\label{thm:rsd_local_convergence}
Let \(P^\star\in\SPD\) be a nondegenerate local minimizer of twice continuously differentiable \(f\) under the
Riemannian metric \(g\), that is,
\[
    \grad^g f(P^\star)=0,
    \qquad
    H_g(P^\star)
    :=
    [\Hess^g f(P^\star)]_{\mathcal B_{P^\star}}
    \succ 0.
\]
Define
\[
    \kappa_g^\star
    :=
    \kappa\!\bigl(H_g(P^\star)\bigr),
    \qquad
    \lambda_{\min,g}^\star
    :=
    \lambda_{\min}\!\bigl(H_g(P^\star)\bigr),
    \qquad
    \lambda_{\max,g}^\star
    :=
    \lambda_{\max}\!\bigl(H_g(P^\star)\bigr).
\]
Let \(\Omega\) be a totally normal neighborhood of \(P^\star\) under the
metric \(g\), with diameter bounded by \(D\). By the nondegeneracy of \(P^\star\) and the continuity of the Riemannian Hessian, there exists a constant
\(c_{\mathrm{sd}}\ge 1\) such that, for every \(P\in\Omega\) and every unit
tangent vector \(U\in T_P\SPD\), the second derivative of \(f\) along the
geodesic \(t\mapsto \Exp_P^g(tU)\) is uniformly bounded as
\[
    \frac{\lambda_{\min,g}^\star}{c_{\mathrm{sd}}}
    \le
    \frac{d^2}{dt^2}
    f\bigl(\Exp_P^g(tU)\bigr)
    \le
    c_{\mathrm{sd}}\lambda_{\max,g}^\star,
    \qquad
    \|U\|_g=1.
\]
Let \(K_{\min,g}\) be a lower bound of the sectional curvature on \(\Omega\),
namely
$K_g(P;\sigma)\ge K_{\min,g}$
for every \(P\in\Omega\) and every two-dimensional tangent plane
\(\sigma\subset T_P\SPD\). Define
\[
    \zeta_g
    :=
    \begin{cases}
    \sqrt{-K_{\min,g}}\,D\,
    \coth\!\bigl(\sqrt{-K_{\min,g}}\,D\bigr),
    & K_{\min,g}<0,\\[4pt]
    1,
    & K_{\min,g}\ge 0.
    \end{cases}
\]
Then Riemannian steepest descent initialized at \(P_0\in\Omega\) with fixed
step size \(\eta=1/(c_{\mathrm{sd}}\lambda_{\max,g}^\star)\) satisfies, for
all \(t\ge 2\),
\[
    d_g^2(P_t,P^\star)
    \le
    c_{\mathrm{sd}}^2D^2\kappa_g^\star
    \left(
        1-
        \min\left\{
            \frac{1}{\zeta_g},
            \frac{1}{c_{\mathrm{sd}}^2\kappa_g^\star}
        \right\}
    \right)^{t-2}.
\]
\end{theorem}

\begin{theorem}[Local convergence of Riemannian trust region {\cite[Theorem~2]{han2021riemannian}}]
\label{thm:rtr_local_convergence}
Under the same setting as in Theorem~\ref{thm:rsd_local_convergence}, let
$   \mathcal H_{P_t}:T_{P_t}\SPD\to T_{P_t}\SPD$
be the symmetric linear operator used in the trust-region model to approximate
the Riemannian Hessian \(\Hess^g f(P_t)\). Assume further on
\(\Omega\),
\[
    \|\mathcal H_{P_t}-\Hess^g f(P_t)\|
    \le
    \ell\|\grad^g f(P_t)\|,
\]
and
\[
    \bigl\|
        \nabla^2(f\circ \Exp^g_{P_t})(U)
        -
        \nabla^2(f\circ \Exp^g_{P_t})(0)
    \bigr\|
    \le
    \rho\|U\|
\]
for some constants \(\ell,\rho>0\). Then running Riemannian trust region from
\(P_0\in\Omega\) yields
\[
    d_g(P_t,P^\star)
    \le
    (2\sqrt{\rho}+\ell)
    \bigl(\kappa_g^\star\bigr)^2
    d_g^2(P_{t-1},P^\star).
\]
\end{theorem}

We now specialize the convergence rate in
Theorem~\ref{thm:rsd_local_convergence} to the \(\AP\) geometries. By
Proposition~\ref{prop:ap_nonnegative_curvature}, every \(\AP\)
geometry has nonnegative sectional curvature. Hence, for the \(\AP\)
metric,
$
    K_{\min,\alpha}\ge 0,
    \quad
    \zeta_g=1.
$
Therefore, unlike the AI geometry, whose sectional curvature can be negative,
the \(\AP\) geometries do not suffer from the curvature penalty
\(\zeta_g>1\) appearing in the RSD convergence factor in
Theorem~\ref{thm:rsd_local_convergence}.

According to the condition-number comparison in Section~\ref{sec:comparison_ai_ap}, if
$
    \kappa(P^\star)
    >
    \sqrt{2}\,\kappa\!\bigl(H_E(P^\star)\bigr),
$
then
$
    \kappa_1^\star
    <
    \kappa_{\mathrm{AI}}^\star.
$
Moreover, for any fixed \(\alpha\neq 1\), if
$
    \kappa(P^\star)
    >
    \left(
        \sqrt{2}\,\kappa\!\bigl(H_E(P^\star)\bigr)
    \right)^{1/|\alpha-1|},
$
then
$
    \kappa_1^\star
    <
    \kappa_\alpha^\star
    \quad
    \text{for any fixed } \alpha\neq 1.
$
Therefore, in sufficiently ill-conditioned regimes, the \(\AP_1\) metric yields a smaller local Hessian condition number than both the AI metric and any fixed \(\AP\) metric with \(\alpha\neq 1\). Together with the nonnegative sectional curvature of the \(\AP\) family, this shows that the \(\AP_1\) geometry combines the absence of the RSD curvature penalty with a Riemannian metric-independent Riemannian Hessian condition number bound. Consequently, it provides sharper local convergence guarantees for RSD and a smaller condition number-dependent local convergence rate for RTR than the AI geometry and the other fixed \(\AP\) geometries with \(\alpha\neq 1\).

\section{Geodesic convexity under AP metric}
\label{sec6}
In this section, we show that geodesic convexity under the \(\mathrm{AP}_{1/2}\) \((\mathrm{BW})\) geometry can be transferred to any \(\mathrm{AP}\) geometry with \(\alpha\neq 0\): in detail, if \(h\) is geodesically convex under \(\mathrm{AP}_{1/2}\), then the function \(F_{\alpha}(P)=h(P^{2\alpha})\) is geodesically convex under \(\mathrm{AP}_{\alpha}\). Firstly, we introduce the definitions of geodesically convex sets and geodesic convexity \cite{sra2015conic}.

\begin{definition}[Geodesic convex set {\cite{sra2015conic}}]
A set \(\mathcal{X} \subseteq \mathcal{M}\) is geodesic convex if for any
\(x,y \in \mathcal{X}\), the distance-minimizing geodesic \(\gamma\) joining
the two points lies entirely in \(\mathcal{X}\).
\end{definition}

\begin{definition}[Geodesic convexity {\cite{sra2015conic}}]
Consider a geodesic convex set \(\mathcal{X} \subseteq \mathcal{M}\). A function
\(f:\mathcal{X}\to\mathbb{R}\) is called geodesic convex if for any
\(x,y\in\mathcal{X}\), the distance-minimizing geodesic \(\gamma\) joining
\(x\) and \(y\) satisfies
\[
    f(\gamma(t)) \leq (1-t)f(x) + t f(y),
    \qquad \forall t\in[0,1].
\]
Function \(f\) is strictly geodesic convex if the equality holds only when
\(t=0,1\).
\end{definition}

We now specialize the above notions to the AP geometry on
\(\SPD\). For each \(\alpha\neq 0\), we denote by
\(\gamma^{\AP_\alpha}_{P,Q}\) the  \(\AP_\alpha\)-geodesic
joining two points \(P,Q\in\SPD\). The following lemma describes how these
geodesics are related across different values of \(\alpha\), and it will be
the key ingredient for transferring geodesic convexity from
\(\AP_{1/2}\) to \(\AP_\alpha\).

\begin{lemma}[Power transformation of \texorpdfstring{\(\AP\)}{AP} geodesics]
\label{lem:power-geodesics}
Let \(\alpha\neq 0\) and let \(P,Q\in\SPD\). Define
\[
    X:=\Phi_\alpha(P)=P^{2\alpha},
    \qquad
    Y:=\Phi_\alpha(Q)=Q^{2\alpha}.
\]
Then \(X,Y\in\SPD\). Moreover, if
\(\gamma^{\AP_\alpha}_{P,Q}:[0,1]\to\SPD\) denotes the
\(\AP_\alpha\)-geodesic joining \(P\) and \(Q\), then
\[
    \Phi_\alpha\!\left(\gamma^{\AP_\alpha}_{P,Q}(t)\right)
    =
    \gamma^{\AP_{1/2}}_{X,Y}(t),
    \qquad t\in[0,1].
\]
Equivalently,
\[
    \left(\gamma^{\AP_\alpha}_{P,Q}(t)\right)^{2\alpha}
    =
    \gamma^{\AP_{1/2}}_{P^{2\alpha},\,Q^{2\alpha}}(t),
    \qquad t\in[0,1].
\]
Hence the map \(\Phi_\alpha:P\mapsto P^{2\alpha}\) maps
\(\AP_\alpha\)-geodesics to \(\AP_{1/2}\)-geodesics.
\end{lemma}

\begin{proof}
By the explicit \(\AP_\alpha\)-geodesic formula by Theorem~8 of \cite{minh2022alpha}, we have
\[
    \gamma^{\AP_\alpha}_{P,Q}(t)
    =
    G(t)^{1/(2\alpha)},
\]
where
\[
\begin{aligned}
    G(t)
    :=
        &(1-t)^2P^{2\alpha}
        +t^2Q^{2\alpha}                                      \\
        &\quad
        +t(1-t)
        \left\{
            (P^{2\alpha}Q^{2\alpha})^{1/2}
            +
            (Q^{2\alpha}P^{2\alpha})^{1/2}
        \right\}.
\end{aligned}
\]
Using \(X=P^{2\alpha}\) and \(Y=Q^{2\alpha}\), this can be written as
\[
\begin{aligned}
    G(t)
    =
        &(1-t)^2X+t^2Y                                      \\
        &\quad
        +t(1-t)
        \left\{
            (XY)^{1/2}
            +
            (YX)^{1/2}
        \right\}.
\end{aligned}
\]
For \(\alpha=1/2\), the Alpha-Procrustes geodesic from \(X\) to \(Y\) is
precisely
\[
    \gamma^{\AP_{1/2}}_{X,Y}(t)=G(t).
\]
Therefore,
\[
    \left(\gamma^{\AP_\alpha}_{P,Q}(t)\right)^{2\alpha}
    =
    G(t)
    =
    \gamma^{\AP_{1/2}}_{X,Y}(t).
\]
Since \(\Phi_\alpha(P)=P^{2\alpha}\), this is equivalent to
\[
    \Phi_\alpha\bigl(\gamma^{\AP_\alpha}_{P,Q}(t)\bigr)
    =
    \gamma^{\AP_{1/2}}_{X,Y}(t).
\]
Substituting back \(X=P^{2\alpha}\) and \(Y=Q^{2\alpha}\) gives
\[
    \left(\gamma^{\AP_\alpha}_{P,Q}(t)\right)^{2\alpha}
    =
    \gamma^{\AP_{1/2}}_{P^{2\alpha},Q^{2\alpha}}(t).
\]
This proves the lemma.
\end{proof}

The geodesic transformation in Lemma~\ref{lem:power-geodesics} allows us to
compare geodesic convexity under different Alpha-Procrustes geometries. This
leads to the following proposition.

\begin{proposition}[Geodesic convexity transfer across Alpha-Procrustes geometries]
\label{prop:transfer-geodesic-convexity}
Let \(\alpha\neq 0\), and let \(h:\SPD\to\mathbb{R}\) be a function. Define
\[
    F_\alpha:\SPD\to\mathbb{R},
    \qquad
    F_\alpha(P):=h(P^{2\alpha}).
\]
Then \(h\) is geodesically convex with respect to the
\(\AP_{1/2}\)-geometry if and only if \(F_\alpha\) is geodesically convex
with respect to the \(\AP_\alpha\)-geometry.
\end{proposition}

\begin{proof}
We first prove the forward implication. Suppose that \(h\) is geodesically
convex with respect to the \(\AP_{1/2}\)-geometry.

Let \(P,Q\in\SPD\), and let
\[
    \gamma^{\AP_\alpha}_{P,Q}(t),
    \qquad t\in[0,1],
\]
be the \(\AP_\alpha\)-geodesic joining \(P\) and \(Q\). Define
\[
    X=P^{2\alpha},
    \qquad
    Y=Q^{2\alpha}.
\]
Since \(P,Q\in\SPD\), we have
\[
    X,Y\in\SPD.
\]
By Lemma~\ref{lem:power-geodesics}, we have
\[
    \left(\gamma^{\AP_\alpha}_{P,Q}(t)\right)^{2\alpha}
    =
    \gamma^{\AP_{1/2}}_{X,Y}(t)
    =
    \gamma^{\AP_{1/2}}_{P^{2\alpha},Q^{2\alpha}}(t).
\]
Therefore,
\[
\begin{aligned}
    F_\alpha\left(\gamma^{\AP_\alpha}_{P,Q}(t)\right)
    &=
    h\left(
        \left(\gamma^{\AP_\alpha}_{P,Q}(t)\right)^{2\alpha}
    \right)                                                   \\
    &=
    h\left(
        \gamma^{\AP_{1/2}}_{X,Y}(t)
    \right).
\end{aligned}
\]
Since \(h\) is geodesically convex under the \(\AP_{1/2}\)-geometry,
\[
    h\left(
        \gamma^{\AP_{1/2}}_{X,Y}(t)
    \right)
    \le
    (1-t)h(X)+t h(Y).
\]
Using \(X=P^{2\alpha}\), \(Y=Q^{2\alpha}\), and
\(F_\alpha(P)=h(P^{2\alpha})\), we obtain
\[
\begin{aligned}
    F_\alpha\left(\gamma^{\AP_\alpha}_{P,Q}(t)\right)
    &\le
    (1-t)h(P^{2\alpha})+t h(Q^{2\alpha})                     \\
    &=
    (1-t)F_\alpha(P)+tF_\alpha(Q).
\end{aligned}
\]
Hence \(F_\alpha\) is geodesically convex with respect to the
\(\AP_\alpha\)-geometry.

Conversely, suppose that \(F_\alpha\) is geodesically convex with respect to
the \(\AP_\alpha\)-geometry. We prove that \(h\) is geodesically convex with
respect to the \(\AP_{1/2}\)-geometry.

Let
\[
    X,Y\in\SPD,
\]
and let
\[
    \gamma^{\AP_{1/2}}_{X,Y}(t),
    \qquad t\in[0,1],
\]
be the \(\AP_{1/2}\)-geodesic joining \(X\) and \(Y\). Define
\[
    P=X^{1/(2\alpha)},
    \qquad
    Q=Y^{1/(2\alpha)}.
\]
Since \(X,Y\in\SPD\), we have
\[
    P,Q\in\SPD.
\]
Moreover,
\[
    X=P^{2\alpha},
    \qquad
    Y=Q^{2\alpha}.
\]
By Lemma~\ref{lem:power-geodesics}, we have,
\[
    \gamma^{\AP_\alpha}_{P,Q}(t)
    =
    \left(
        \gamma^{\AP_{1/2}}_{X,Y}(t)
    \right)^{1/(2\alpha)}.
\]
Therefore,
\[
\begin{aligned}
    F_\alpha\left(
        \gamma^{\AP_\alpha}_{P,Q}(t)
    \right)
    &=
    F_\alpha\left(
        \left(
            \gamma^{\AP_{1/2}}_{X,Y}(t)
        \right)^{1/(2\alpha)}
    \right)                                                   \\
    &=
    h\left(
        \gamma^{\AP_{1/2}}_{X,Y}(t)
    \right).
\end{aligned}
\]
Since \(F_\alpha\) is geodesically convex under the \(\AP_\alpha\)-geometry,
\[
    F_\alpha\left(
        \gamma^{\AP_\alpha}_{P,Q}(t)
    \right)
    \le
    (1-t)F_\alpha(P)+tF_\alpha(Q).
\]
Using \(P=X^{1/(2\alpha)}\), \(Q=Y^{1/(2\alpha)}\), and
\(F_\alpha(P)=h(P^{2\alpha})\), we obtain
\[
\begin{aligned}
    h\left(\gamma^{\AP_{1/2}}_{X,Y}(t)\right)
    &\le
    (1-t)F_\alpha(P)+tF_\alpha(Q)                             \\
    &=
    (1-t)h(P^{2\alpha})+t h(Q^{2\alpha})                       \\
    &=
    (1-t)h(X)+t h(Y).
\end{aligned}
\]
Therefore \(h\) is geodesically convex with respect to the
\(\AP_{1/2}\)-geometry.
\end{proof}

As a direct consequence of Proposition~\ref{prop:transfer-geodesic-convexity},
known geodesic convexity results under the \(\AP_{1/2}\) \((\mathrm{BW})\)
geometry immediately yields new geodesic convexity results under the
\(\AP_\alpha\)-geometry. We summarize several representative examples below.

\begin{corollary}[Examples of transferred geodesic convexity]
\label{cor:examples-transferred-convexity}
Let \(A\succeq 0\) and \(\alpha\neq 0\). Suppose that the functions
\(h_i:\SPD\to\mathbb{R}\), \(i=1,2,3\), defined by
\[
    h_1(P):=\tr(AP),
    \qquad
    h_2(P):=\tr(PAP),
    \qquad
    h_3(P):=-\log\det P,
    \qquad P\in\SPD,
\]
are geodesically convex with respect to the \(\AP_{1/2}\) \((\mathrm{BW})\)
geometry \cite{han2021riemannian}. Then the corresponding transformed functions
\(F_{i,\alpha}:\SPD\to\mathbb{R}\), \(i=1,2,3\), given by
\[
    F_{1,\alpha}(P):=\tr(AP^{2\alpha}),
    \qquad
    F_{2,\alpha}(P):=\tr(P^{2\alpha}AP^{2\alpha}),
    \qquad
    F_{3,\alpha}(P):=-\log\det(P^{2\alpha}),
\]
are geodesically convex with respect to the \(\AP_\alpha\)-geometry.
\end{corollary}

\section{Experiments} \label{sec7}
In this section, we empirically evaluate the performance of optimization algorithms by comparing them under different Riemannian geometries across several problems. In addition to the AP geometry, we also include the affine-invariant (AI) geometry in our experiments. 
The AI geometry is one of the most widely studied Riemannian structures on $\SPD(n)$. It is defined by the Riemannian metric
\[
g^{\mathrm{AI}}_P(X,Y)=\mathrm{tr}\!\left(P^{-1} X P^{-1} Y\right),
\qquad
P\in\SPD(n),\; X,Y\in T_P\SPD(n).
\]
It has been observed that the BW metric is often better suited for optimizing ill-conditioned SPD matrices than the AI metric \cite{han2021riemannian}.

We present convergence mainly in terms of the distance to the solution \(P^\star\) whenever applicable.
The distance is measured by the Frobenius norm, i.e., \(\|P_t - P^\star\|_F\).
We initialize the algorithms with the identity matrix for all the metrics.
Moreover, we report experimental results for both the RSD method and the RTR method.
For the RTR method, the trust-region subproblem is approximately solved by the truncated conjugate gradient (tCG) method. For all methods, we use the stopping criterion
$\|\nabla f(P_k)\|_F < 10^{-6}.$ All experiments were conducted on an Apple M2 Max CPU.

\subsection{Weighted Least Squares}
We consider the weighted least-squares problem on the SPD manifold
\begin{equation}
\label{eq:wls_spd}
\min_{P \in \SPD(n)} 
f(P)
=
\frac12 \| A \odot P - B \|_F^2,
\end{equation}
where $\odot$ denotes the Hadamard (elementwise) product and 
$A,B\in\Sym(n)$ are given weight and target matrices, respectively.
Since $f$ is quadratic in $P$, its Euclidean gradient and Hessian are
\begin{equation}
\nabla f(P)=(A \odot P - B)\odot A,
\qquad
\nabla^2 f(P)[U]=A \odot U \odot A,
\quad U\in\Sym(n).
\end{equation}
We consider the weighted least-squares problem with
\[
A=\mathbf{1}_n\mathbf{1}_n^\top,
\qquad
B=A\odot P^\star
\]
and study two spectral regimes: a low-condition-number case with $\kappa(P^\star)=10$ and a high-condition-number case with $\kappa(P^\star)=10^4$.

For both RSD and RTR, the target matrix $P^\star\in\SPD(50)$ is generated as
\[
P^\star = Q\,\mathrm{diag}(\lambda_1,\ldots,\lambda_{50})\,Q^\top,
\]
where $Q$ is a random orthogonal matrix.
For the RSD experiments, the eigenvalues are chosen as
\[
\lambda_i
=
\exp\!\left(
-\frac{(i-1)\log \kappa(P^\star)}{49}
\right),
\qquad i=1,\ldots,50,
\]
which yields an exponentially decaying spectrum from $1$ to $\kappa(P^\star)^{-1}$.
For the RTR experiments, the eigenvalues of $P^\star$ are chosen by geometric interpolation between $\kappa(P^\star)^{1/2}$ and $\kappa(P^\star)^{-1/2}$, where each $\lambda_i$ is given as
\[
\lambda_i
=
\kappa(P^\star)^{1/2}
\left(\kappa(P^\star)^{-1}\right)^{\frac{i-1}{49}},
\qquad i=1,\ldots,50.
\]
The maximum number of iterations is set to $k_{\max}=200$ for both RSD and RTR.

Figure~\ref{fig:wls_rsd_rtr_all} compares the performance of RSD and RTR under different Riemannian metrics. For RSD, the metric with \(\alpha=1\) yields the fastest convergence in both the low- and high-condition-number settings, and its advantage becomes more pronounced when the condition number of \(P^\ast\) is large. For RTR, the same trend is observed: \(\alpha=1\) converges in the smallest number of iterations and remains the most robust choice as the condition number increases. Table~\ref{tab:wls_rsd_rtr_summary} further shows that, as the condition number increases, the metric with \(\alpha=1\) is also superior in terms of both iteration count and runtime.

\begin{figure*}[t]
    \centering
    \captionsetup{font=small}

    \begin{subfigure}[t]{0.48\textwidth}
        \centering
        \includegraphics[width=\textwidth]{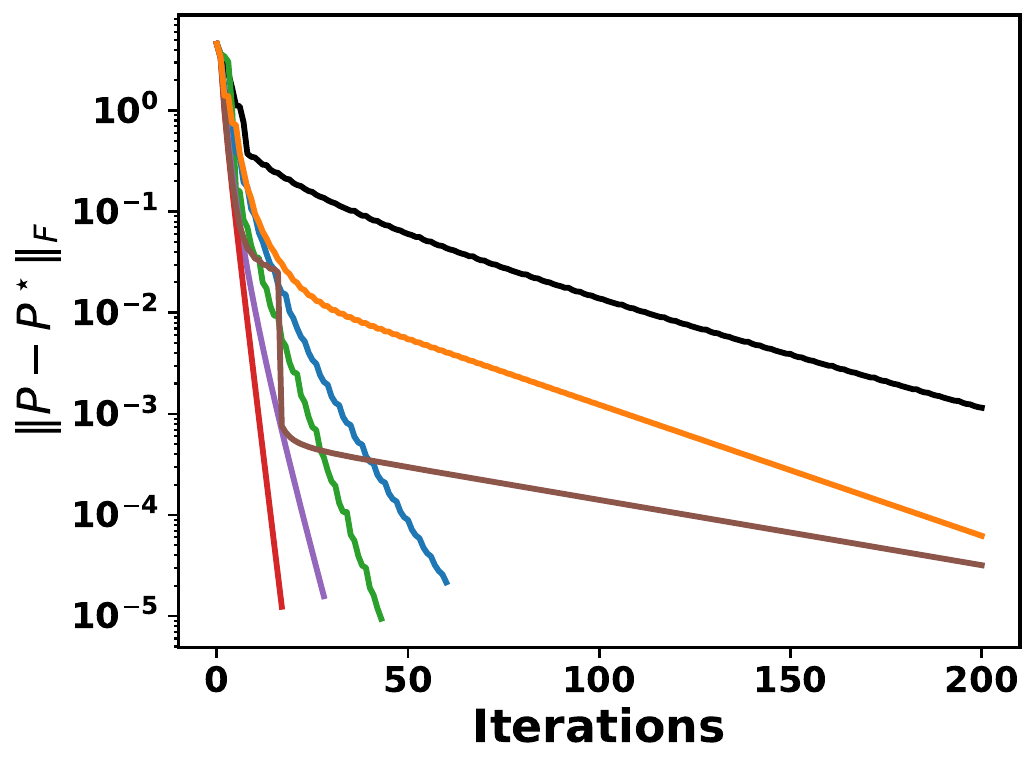}
        \caption{RSD, $\kappa(P^\star)=10^1$.}
    \end{subfigure}
    \hfill
    \begin{subfigure}[t]{0.48\textwidth}
        \centering
        \includegraphics[width=\textwidth]{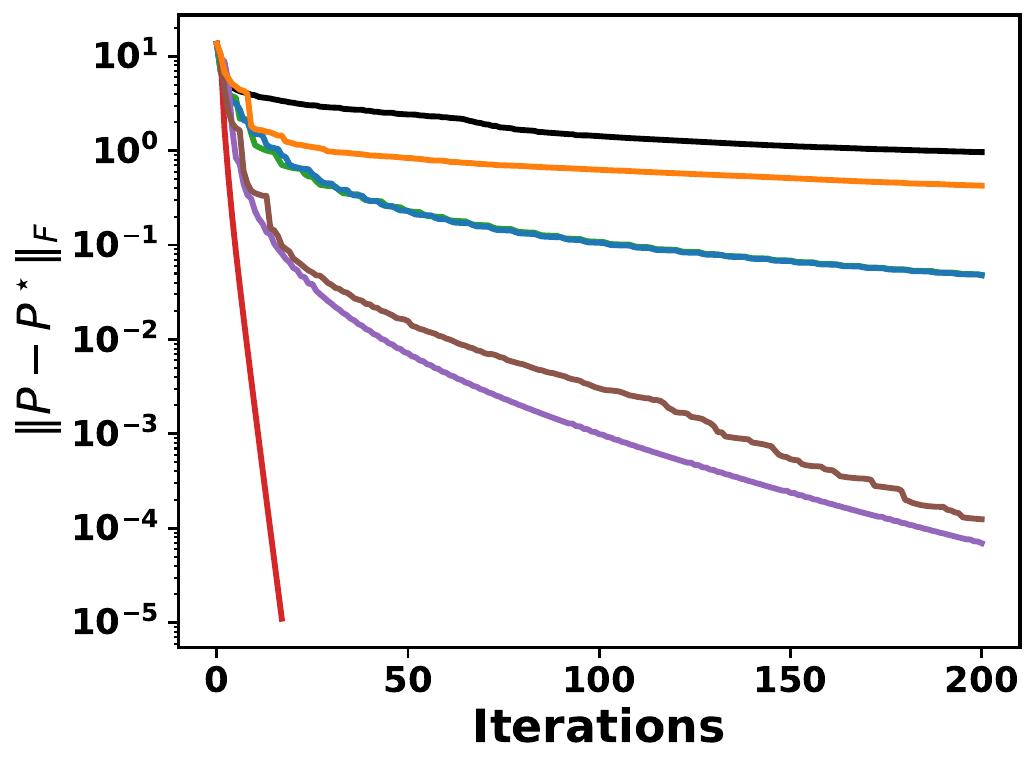}
        \caption{RSD, $\kappa(P^\star)=10^4$.}
    \end{subfigure}

    \vspace{0.6em}

    \begin{subfigure}[t]{0.48\textwidth}
        \centering
        \includegraphics[width=\textwidth]{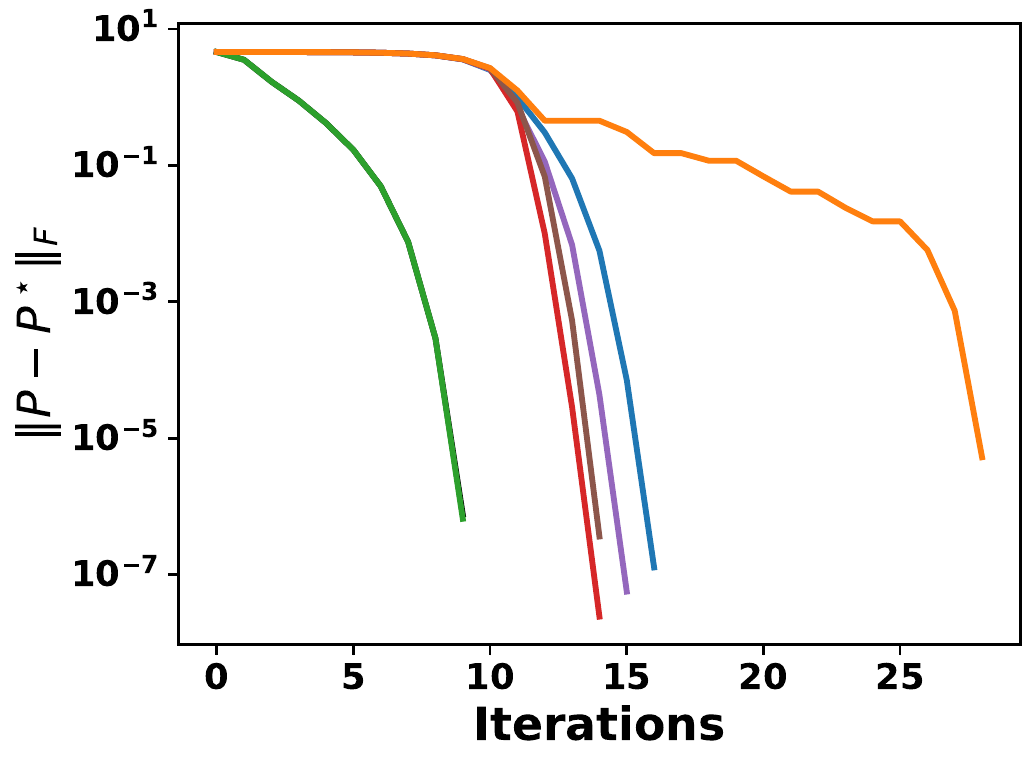}
        \caption{RTR, $\kappa(P^\star)=10^1$.}
    \end{subfigure}
    \hfill
    \begin{subfigure}[t]{0.48\textwidth}
        \centering
        \includegraphics[width=\textwidth]{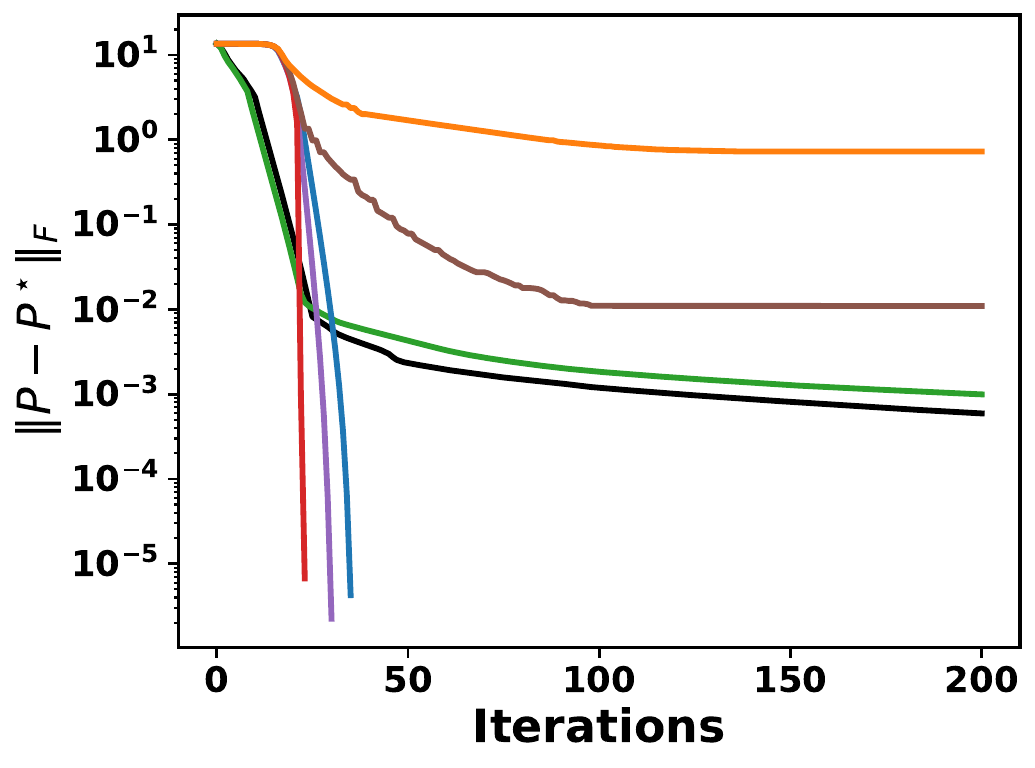}
        \caption{RTR, $\kappa(P^\star)=10^4$.}
    \end{subfigure}


        \includegraphics[width=0.95\textwidth]{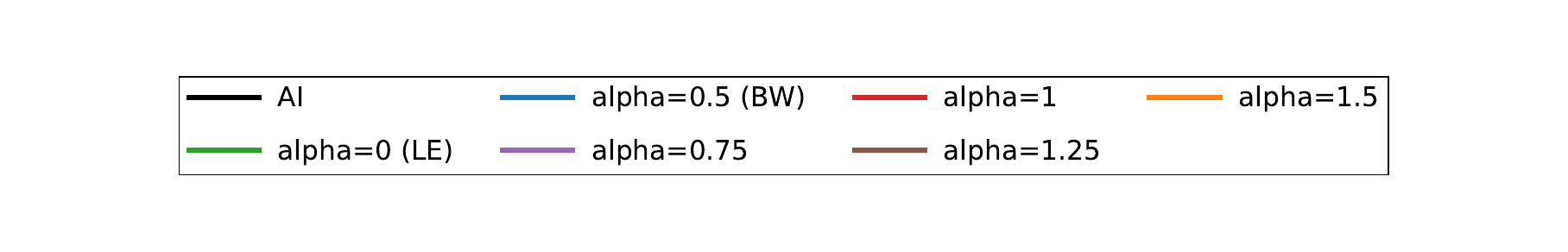}


    \caption{Convergence curves for the weighted least-squares problem on $\SPD(n)$. The first row shows the RSD results for the dense case under low and high condition numbers, while the second row shows the corresponding RTR results.}
    \label{fig:wls_rsd_rtr_all}
\end{figure*}

\begin{table*}[t]
\centering
\caption{Summary of RSD and RTR results for the weighted least-squares problem on \(\SPD(n)\).
For each setting, we report the iteration count and runtime.
The best result in each block is underlined.}
\label{tab:wls_rsd_rtr_summary}
\resizebox{\textwidth}{!}{
\begin{tabular}{l rr rr rr rr}
\toprule
\multirow{2}{*}{metric}
& \multicolumn{2}{c}{\(\mathrm{RTR},\ \kappa(P^\star)=10\)}
& \multicolumn{2}{c}{\(\mathrm{RTR},\ \kappa(P^\star)=10^4\)}
& \multicolumn{2}{c}{\(\mathrm{RSD},\ \kappa(P^\star)=10\)}
& \multicolumn{2}{c}{\(\mathrm{RSD},\ \kappa(P^\star)=10^4\)} \\
\cmidrule(lr){2-3}\cmidrule(lr){4-5}\cmidrule(lr){6-7}\cmidrule(lr){8-9}
& \#iter & time (s)
& \#iter & time (s)
& \#iter & time (s)
& \#iter & time (s) \\
\midrule
AI
& \underline{9} & \underline{0.193}
& 200 & 46.484
& 200 & 0.300
& 200 & 0.618 \\

LE (\(\alpha=0\))
& \underline{9} & 0.402
& 200 & 84.357
& 43 & 0.030
& 200 & 0.263 \\

BW (\(\alpha=0.5\))
& 16 & 0.265
& 35 & 2.494
& 60 & 0.182
& 200 & 0.885 \\

\(\alpha=0.75\)
& 15 & 0.287
& 30 & 1.136
& 28 & 0.069
& 200 & 0.599 \\

\(\alpha=1\)
& 14 & 0.235
& \underline{23} & \underline{0.341}
& \underline{17} & \underline{0.042}
& \underline{17} & \underline{0.042} \\

\(\alpha=1.25\)
& 14 & 0.264
& 200 & 2.314
& 200 & 0.505
& 200 & 1.162 \\

\(\alpha=1.5\)
& 28 & 0.759
& 200 & 2.073
& 200 & 0.708
& 200 & 1.434 \\
\bottomrule
\end{tabular}}
\end{table*}

\subsection{Trace Regression}
Next, we consider trace regression on the SPD manifold with rank-one sensing matrices \(A_i = a_i a_i^\top\), where \(a_i \in \mathbb{R}^n\). The problem can be written as
\begin{equation}
\label{eq:tr_spd}
\min_{P\in\SPD(n)}
f(P)
\;=\;
\frac{1}{2m}\sum_{i=1}^m
\bigl(\tr(A_i P)-y_i\bigr)^2
\;=\;
\frac{1}{2m}\sum_{i=1}^m
\bigl(a_i^\top P a_i-y_i\bigr)^2.
\end{equation}
Letting the residual
$r_i(P):=\tr(A_iP)-y_i$ $i=1,\dots,m$,
the Euclidean gradient of \(f\) is given by
\begin{equation}
\label{eq:tr_egrad}
\nabla f(P)
=
\frac{1}{m}\sum_{i=1}^m r_i(P)\,A_i
=
\frac{1}{m}\sum_{i=1}^m r_i(P)\,a_i a_i^\top .
\end{equation}
The Euclidean Hessian is
defined for \(U\in\Sym(n)\) by
\begin{equation}
\label{eq:tr_ehess}
\nabla^2 f(P)[U]
=
\frac{1}{m}\sum_{i=1}^m \langle A_i,U\rangle\,A_i
=
\frac{1}{m}\sum_{i=1}^m \bigl(a_i^\top U a_i\bigr)\,a_i a_i^\top.
\end{equation}
In the experiments, the sensing matrices are generated as independent rank-one Wishart matrices:
\begin{equation}
\label{eq:tr_Ai}
A_i = a_i a_i^\top,
\qquad
a_i \overset{\mathrm{i.i.d.}}{\sim} \mathcal N(0,I_n),
\qquad i=1,\dots,m,
\end{equation}
and the responses are generated according to the noisy observation model
\begin{equation}
\label{eq:tr_y}
y_i
=
a_i^\top P^\star a_i
+
\sigma \varepsilon_i,
\qquad
\varepsilon_i \overset{\mathrm{i.i.d.}}{\sim} \mathcal N(0,0.01),
\qquad i=1,\dots,m.
\end{equation}
The matrix \(P^\star \in \SPD(n)\) is generated as
\begin{equation}
\label{eq:tr_pstar}
P^\star
=
Q \,\diag(\lambda_1,\dots,\lambda_n)\, Q^\top,
\end{equation}
where \(Q\) is a random orthogonal matrix. In the RSD method, the eigenvalues are chosen by geometric interpolation from \(1\) to \(\kappa(P^\star)^{-1}\), namely,
\begin{equation}
\label{eq:tr_pstar_eigs}
\lambda_i
=
\left(\kappa(P^\star)^{-1}\right)^{\frac{i-1}{n-1}},
\qquad i=1,\dots,n.
\end{equation}
In the RTR method, by contrast, the eigenvalues are chosen by geometric interpolation between \(\kappa(P^\star)^{1/2}\) and \(\kappa(P^\star)^{-1/2}\), namely,
\begin{equation}
\label{eq:tr_pstar_eigs_rtr}
\lambda_i
=
\kappa(P^\star)^{\frac{1}{2}-\frac{i-1}{n-1}},
\qquad i=1,\dots,n.
\end{equation}
The maximum number of iterations is set to $k_{\max}=800$ for RSD and $k_{\max}=400$ for RTR, respectively. The maximum number of iterations is set to \(k_{\max}=800\) for RSD and \(k_{\max}=400\) for RTR.

Results are shown in Figure~\ref{fig:trace_rsd_rtr_all}, which compares the performance of RSD and RTR under different Riemannian metrics. For RSD, the metric with \(\alpha=1\) yields the fastest convergence in both the low- and high-condition-number settings, and its advantage becomes more pronounced when the condition number of \(P^\ast\) is large. For RTR, a similar trend is observed in the high-condition-number regime. When \(\kappa(P^\ast)=10\), several metrics perform competitively, and \(\alpha=1\) is not the fastest choice. However, when the condition number increases to \(\kappa(P^\ast)=10^3\), \(\alpha=1\) clearly becomes the most robust and best-performing choice, requiring the fewest iterations to converge, whereas some other metrics slow down significantly or fail to make sufficient progress within the iteration budget. Table~\ref{tab:trace_rsd_rtr_summary} further confirms that, as the condition number increases, the metric with \(\alpha=1\) is also superior in terms of both iteration count and runtime.

\begin{figure*}[t]
    \centering
    \captionsetup{font=small}

    \begin{subfigure}[t]{0.48\textwidth}
        \centering
        \includegraphics[width=\textwidth]{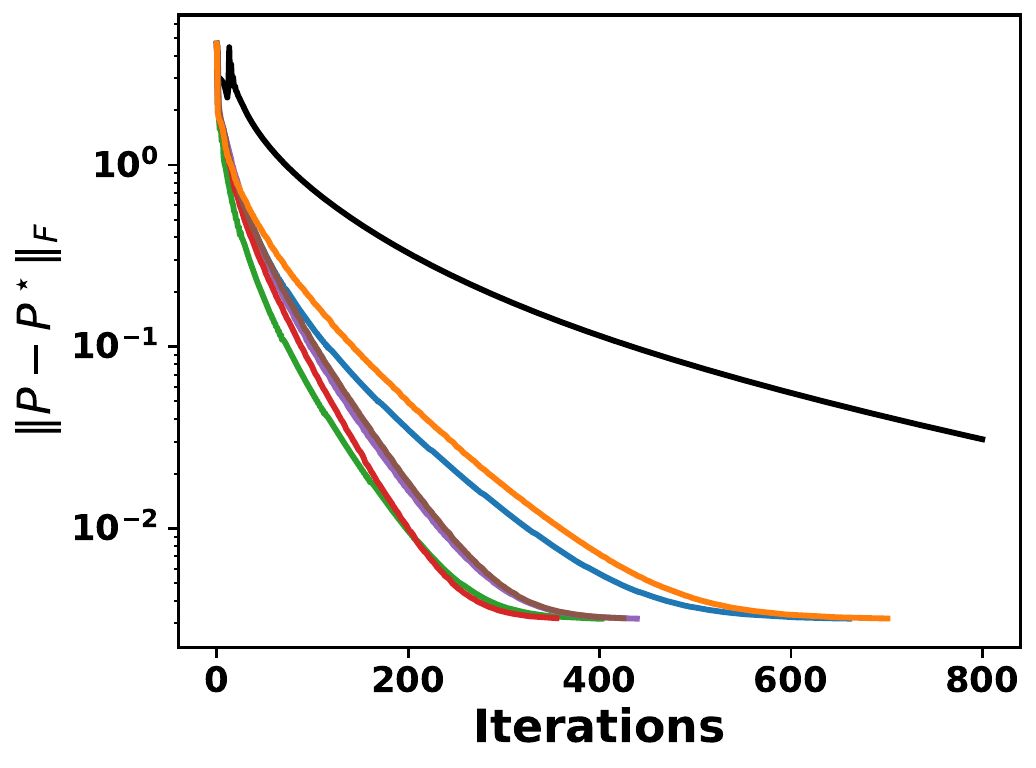}
        \caption{RSD, $\kappa(P^\star)=10^1$}
    \end{subfigure}
    \hfill
    \begin{subfigure}[t]{0.48\textwidth}
        \centering
        \includegraphics[width=\textwidth]{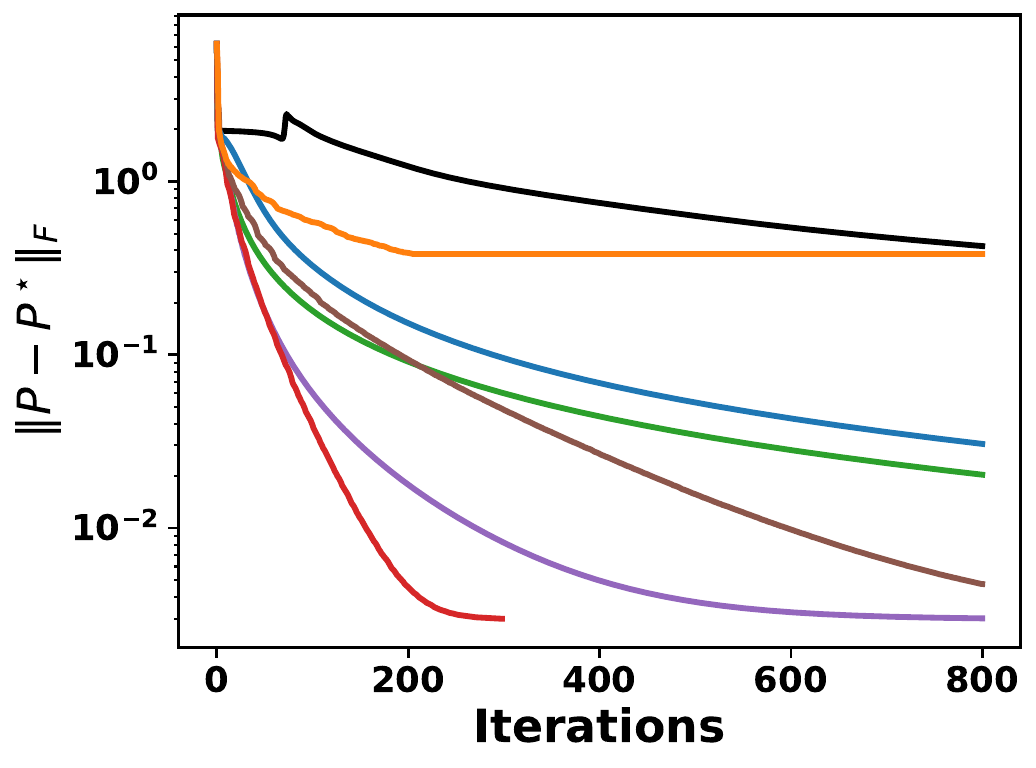}
        \caption{RSD, $\kappa(P^\star)=10^3$}
    \end{subfigure}

    \vspace{0.6em}

    \begin{subfigure}[t]{0.48\textwidth}
        \centering
        \includegraphics[width=\textwidth]{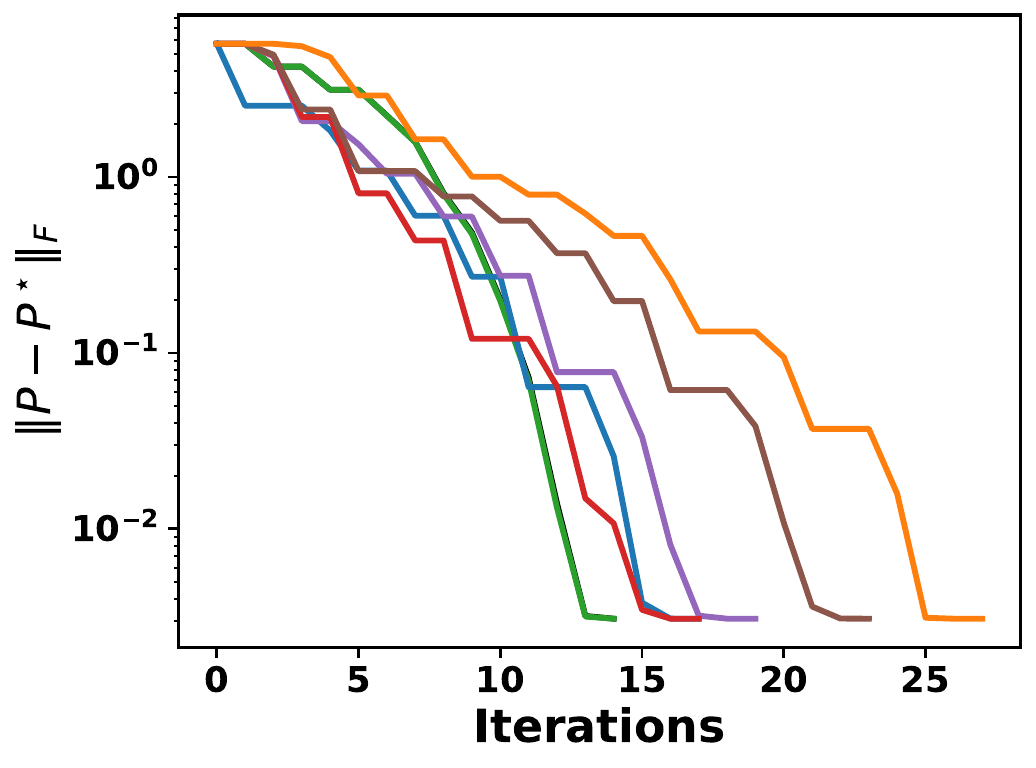}
        \caption{RTR, $\kappa(P^\star)=10^1$}
    \end{subfigure}
    \hfill
    \begin{subfigure}[t]{0.48\textwidth}
        \centering
        \includegraphics[width=\textwidth]{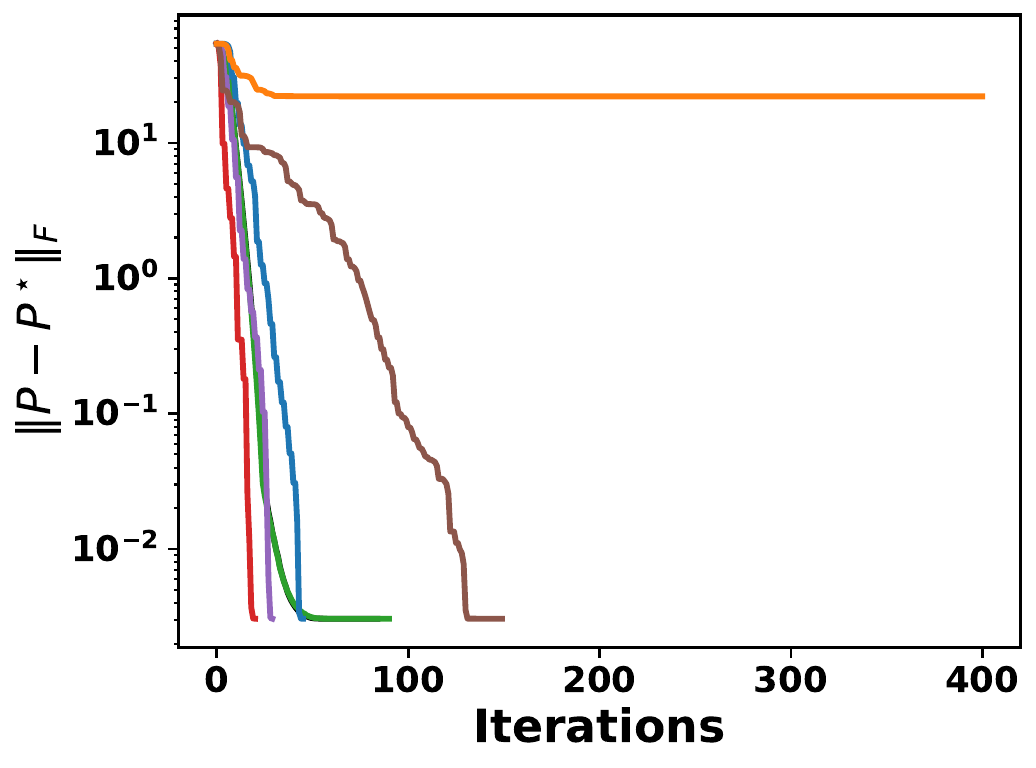}
        \caption{RTR, $\kappa(P^\star)=10^3$}
    \end{subfigure}


        \includegraphics[width=0.95\textwidth]{sylvester_custom_legend_2rows.pdf}


    \caption{Convergence curves for the trace regression problem on $\SPD(n)$ under different condition numbers of the target solution $P^\star$. The first row shows the RSD results, while the second row shows the corresponding RTR results.}
    \label{fig:trace_rsd_rtr_all}
\end{figure*}

\begin{table*}[t]
\centering
\caption{Summary of RSD and RTR results for the trace regression problem on \(\SPD(n)\).
For each setting, we report the iteration count and runtime.
The best result in each block is underlined.}
\label{tab:trace_rsd_rtr_summary}
\resizebox{\textwidth}{!}{
\begin{tabular}{l rr rr rr rr}
\toprule
\multirow{2}{*}{metric}
& \multicolumn{2}{c}{\(\mathrm{RTR},\ \kappa(P^\star)=10\)}
& \multicolumn{2}{c}{\(\mathrm{RTR},\ \kappa(P^\star)=10^3\)}
& \multicolumn{2}{c}{\(\mathrm{RSD},\ \kappa(P^\star)=10\)}
& \multicolumn{2}{c}{\(\mathrm{RSD},\ \kappa(P^\star)=10^3\)} \\
\cmidrule(lr){2-3}\cmidrule(lr){4-5}\cmidrule(lr){6-7}\cmidrule(lr){8-9}
& \#iter & time (s)
& \#iter & time (s)
& \#iter & time (s)
& \#iter & time (s) \\
\midrule
AI
& \underline{14} & \underline{0.560}
& 84 & 27.086
& 800 & 2.006
& 800 & 1.997 \\

LE (\(\alpha=0\))
& \underline{14} & 1.215
& 90 & 53.021
& 396 & 1.824
& 800 & 1.660 \\

BW (\(\alpha=0.5\))
& 17 & 1.010
& 45 & 4.552
& 667 & 2.773
& 800 & 3.307 \\

\(\alpha=0.75\)
& 19 & 1.180
& 29 & 1.364
& 408 & 1.754
& 800 & 3.311 \\

\(\alpha=1\)
& 17 & 0.934
& \underline{20} & \underline{0.689}
& \underline{339} & \underline{1.467}
& \underline{283} & \underline{1.232} \\

\(\alpha=1.25\)
& 19 & 1.056
& 82 & 6.034
& 409 & 1.788
& 800 & 3.453 \\

\(\alpha=1.5\)
& 27 & 1.311
& 400 & 5.541
& 679 & 2.956
& 800 & 3.481 \\
\bottomrule
\end{tabular}}
\end{table*}

\subsection{Sylvester Equation}
Finally, we consider the following convex quadratic optimization problem over the SPD manifold:
\begin{equation}
\label{eq:sylvester_spd}
\min_{P\in\SPD(n)}
f(P)
=
\frac12 \,\langle P, A P + P B\rangle_F
-
\langle C, P\rangle_F,
\end{equation}
where \(A,B\in\SPD(n)\), and \(C\in\Sym(n)\).
Moreover, the Euclidean gradient and Hessian of \(f\) are given by
\begin{align}
\nabla f(P)
&=
A P + P B - C,
\\
\nabla^2 f(P)[U]
&=
A U + U B,
\qquad U\in\Sym(n).
\end{align}
Given \(A\) and \(B\), we choose
\begin{equation}
C = A P^\star + P^\star B,
\end{equation}
so that \(P^\star\) is the unique minimizer of \(f\).

In our experiments, we compare RSD and RTR methods under several metrics on \(\SPD(n)\). For both algorithms, the matrices \(A\) and \(B\) are generated as SPD matrices with prescribed condition numbers. Specifically, they are constructed in the form
\[
A = Q_A \diag(\mu_1,\dots,\mu_n) Q_A^\top,
\qquad
B = Q_B \diag(\nu_1,\dots,\nu_n) Q_B^\top,
\]
where \(\{\mu_i\}\) and \(\{\nu_i\}\) are geometrically distributed eigenvalues. In the RSD experiments, we set \(n=50\) and
$\kappa(A)=\kappa(B)=10^4.$
In the RTR experiments, we instead set \(n=60\) and $\kappa(A)=30, \kappa(B)=20.$

The target optimizer \(P^\star\in\SPD(n)\) is also generated spectrally. In the RSD experiments, it is constructed as
\begin{equation}
\label{eq:sylvester_pstar_rsd}
P^\star
=
Q \diag(\lambda_1,\ldots,\lambda_n) Q^\top,
\qquad
\lambda_i
=
\exp\!\Bigl(
-(i-1)\frac{\log \kappa(P^\star)}{n-1}
\Bigr),
\quad i=1,\dots,n.
\end{equation} 
In the RTR experiments, \(P^\star\) is constructed in a slightly more general form:
\begin{equation}
\label{eq:sylvester_pstar_rtr}
P^\star
=
Q \diag(\lambda_1,\ldots,\lambda_n) Q^\top,
\qquad
\lambda_i
=
c\,\tau^{\,\frac{s}{2}-\,s\frac{i-1}{n-1}},
\quad i=1,\ldots,n.
\end{equation}
where \(\tau>0\) is a prescribed spectral-scaling parameter, \(c>0\) is a center parameter, and \(s>0\) is a stretch parameter. Consequently, $\kappa(P^\star)=\tau^s.$
Unless otherwise specified, we use \(c=1\) and \(s=1.5\) in the RTR experiments. In both cases, once \(P^\star\) is fixed, we define $C = A P^\star + P^\star B,$
which ensures that \(P^\star\) is the unique minimizer of \eqref{eq:sylvester_spd}. The maximum number of iterations is set to \(20000\) for the RSD method and to \(100\) for the RTR method, respectively.

Figure~\ref{fig:sylvester_rsd_rtr_all} compares the performance of RSD and RTR under different Riemannian metrics for the Sylvester equation. When \(\kappa(P^\star)=10\), the metric with \(\alpha=1.5\) converges fastest, while \(\alpha=1\) is also clearly competitive. However, as the condition number increases, the metric with \(\alpha=1\) becomes the best-performing choice.
For RTR, the picture is slightly different in the low-condition-number case. When \(\kappa(P^\star)=10\) and \(10^2\), the metric with \(\alpha=0.75\) attains the smallest iteration count. As the condition number increases further, the advantage shifts toward \(\alpha=1\). 

Tables~\ref{tab:ly_rsd_summary} and~\ref{tab:ly_rtr_summary} quantitatively confirm these observations. In particular, it shows that, although \(\alpha=1\) is not always the best choice in the easiest settings, it becomes the most robust and efficient metric as the condition number increases, especially in terms of both iteration count and runtime in the practically important ill-conditioned regime.

\begin{figure*}[t]
    \centering
    \captionsetup{font=small}

    \begin{subfigure}[t]{0.235\textwidth}
        \centering
        \includegraphics[width=\textwidth]{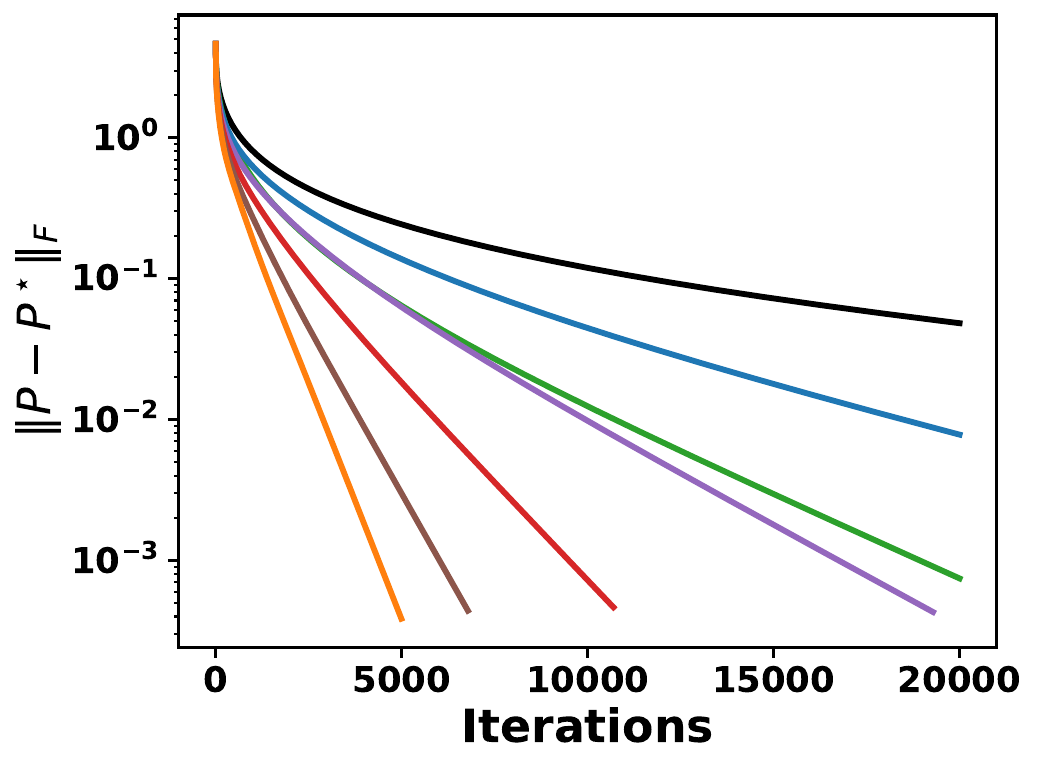}
        \caption{RSD, $\kappa(P^\star)=10^1$}
    \end{subfigure}
    \hfill
    \begin{subfigure}[t]{0.235\textwidth}
        \centering
        \includegraphics[width=\textwidth]{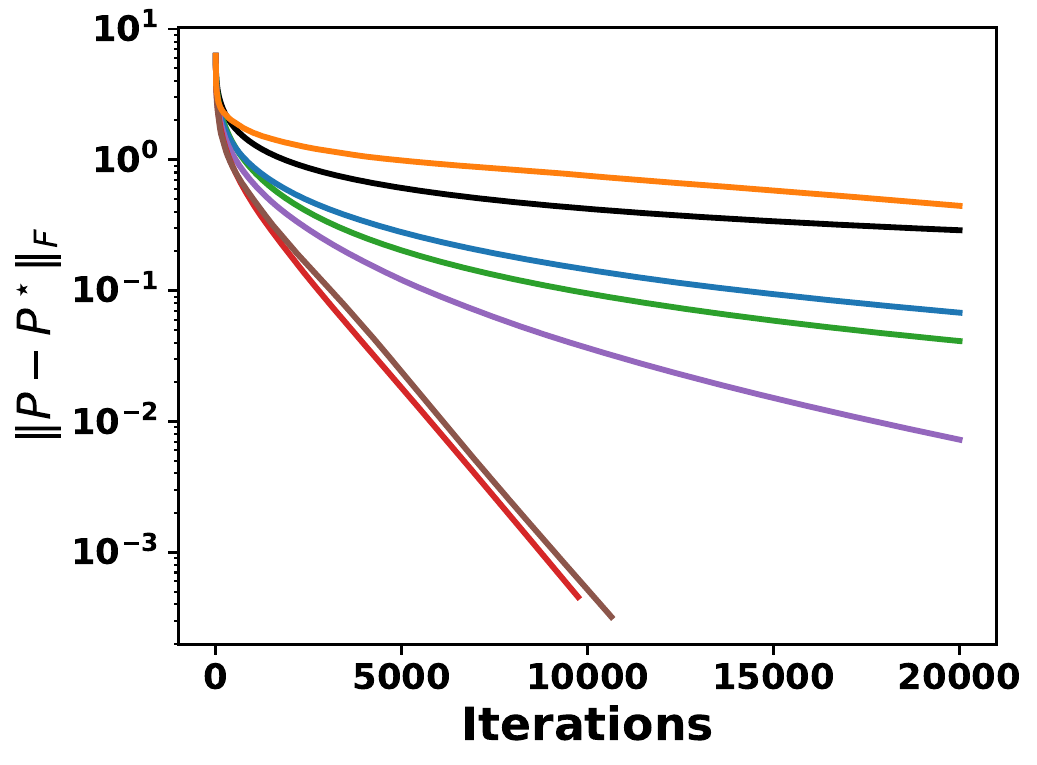}
        \caption{RSD, $\kappa(P^\star)=10^3$}
    \end{subfigure}
    \hfill
    \begin{subfigure}[t]{0.235\textwidth}
        \centering
        \includegraphics[width=\textwidth]{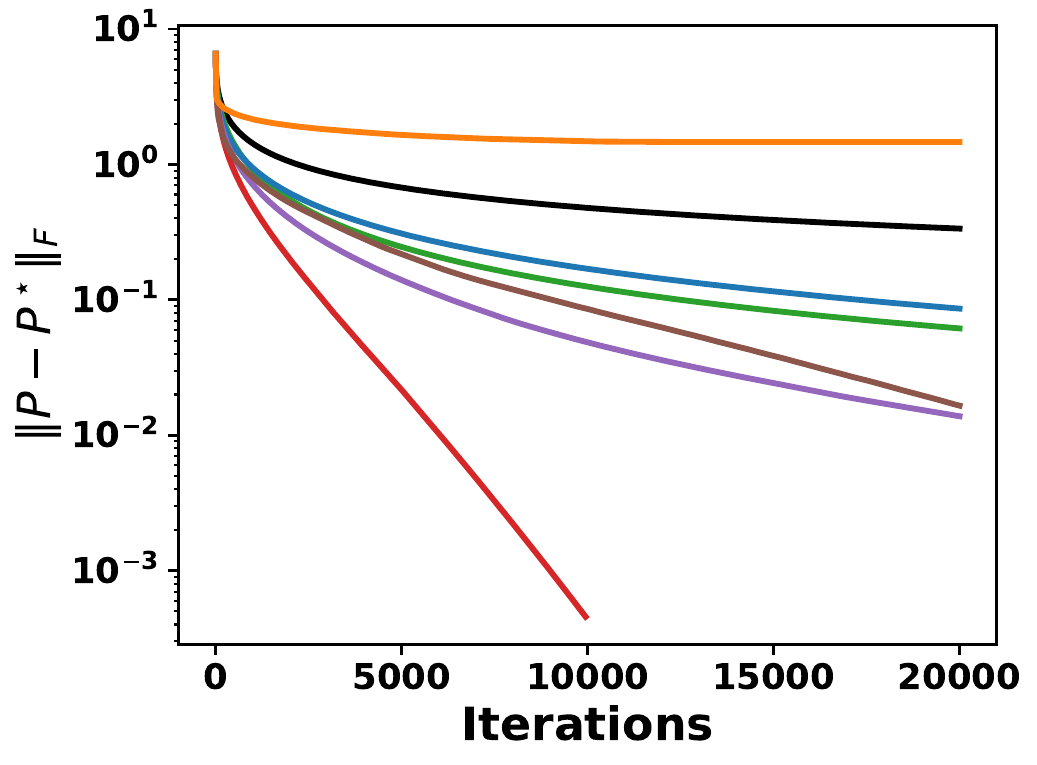}
        \caption{RSD, $\kappa(P^\star)=10^5$}
    \end{subfigure}
    \hfill
    \begin{subfigure}[t]{0.235\textwidth}
        \centering
        \includegraphics[width=\textwidth]{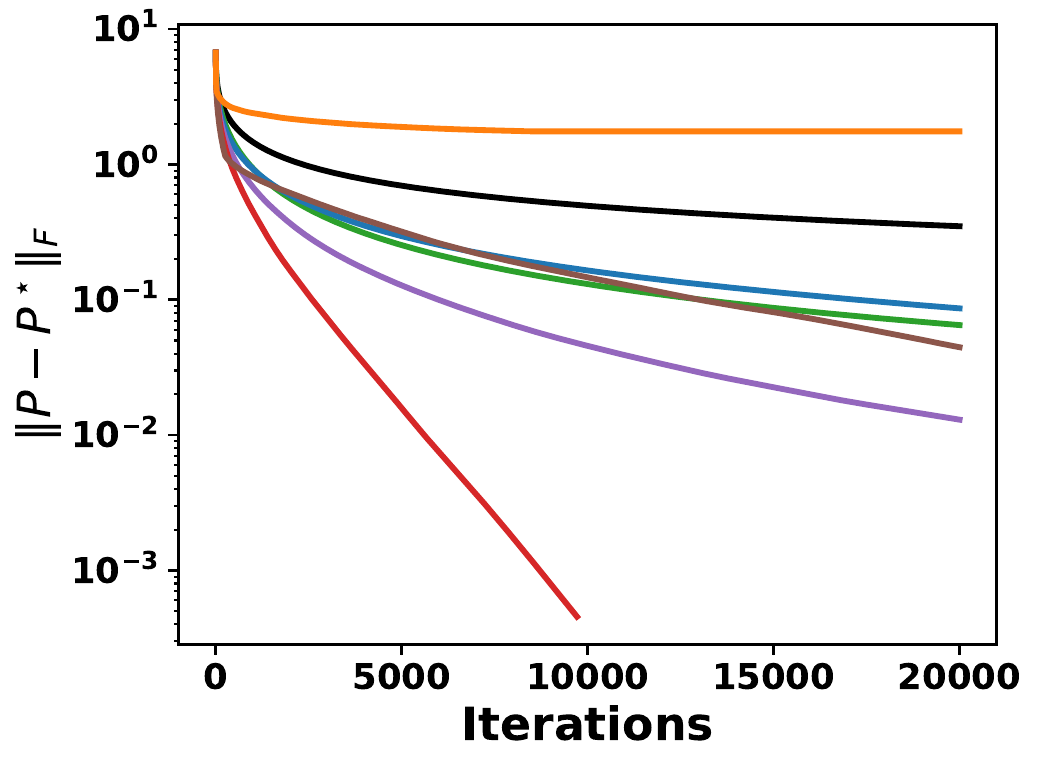}
        \caption{RSD, $\kappa(P^\star)=10^7$}
    \end{subfigure}

    \vspace{0.5em}

    \begin{subfigure}[t]{0.235\textwidth}
        \centering
        \includegraphics[width=\textwidth]{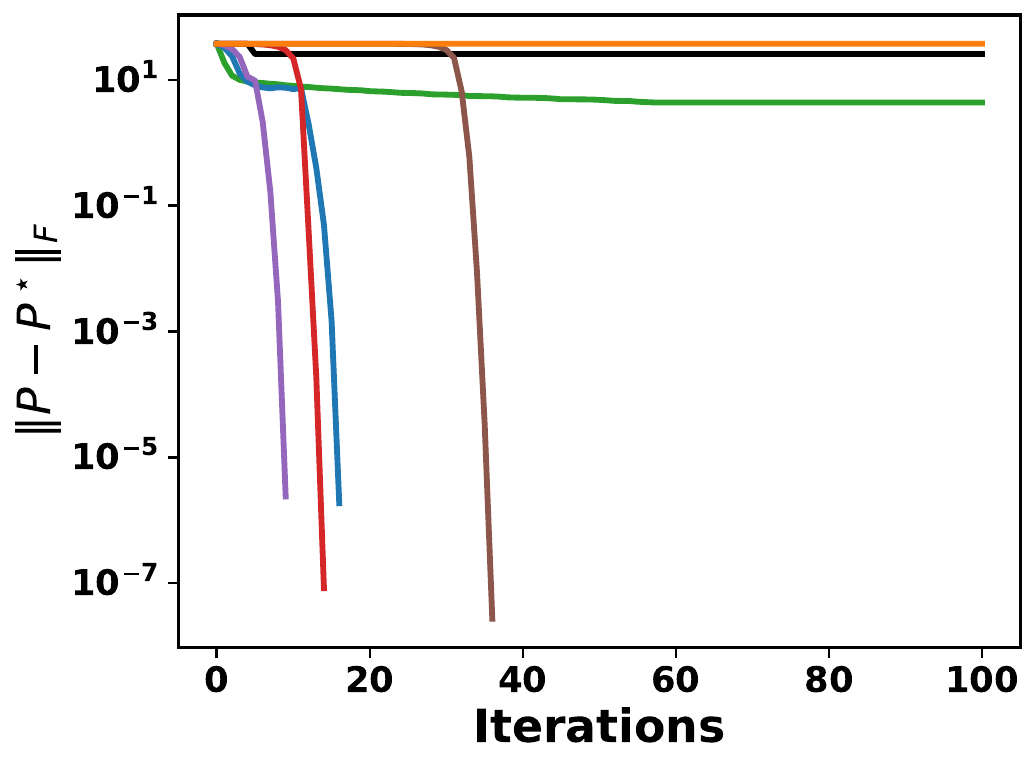}
        \caption{RTR, $\kappa(P^\star)=10^1$}
    \end{subfigure}
    \hfill
    \begin{subfigure}[t]{0.235\textwidth}
        \centering
        \includegraphics[width=\textwidth]{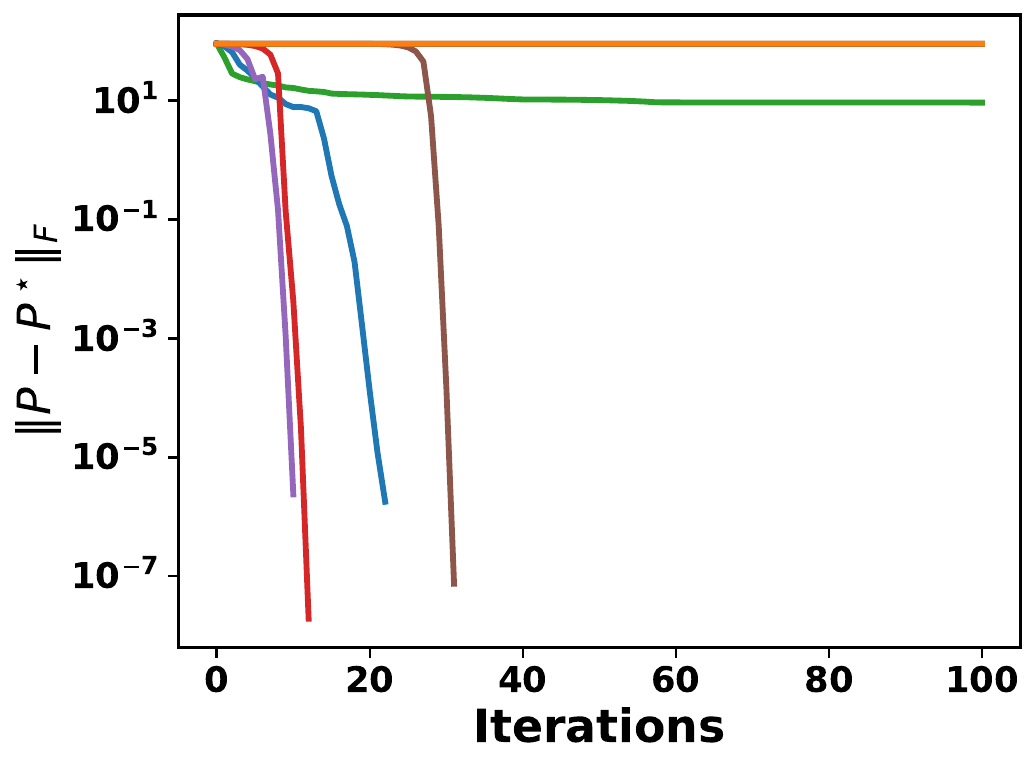}
        \caption{RTR, $\kappa(P^\star)=10^2$}
    \end{subfigure}
    \hfill
    \begin{subfigure}[t]{0.235\textwidth}
        \centering
        \includegraphics[width=\textwidth]{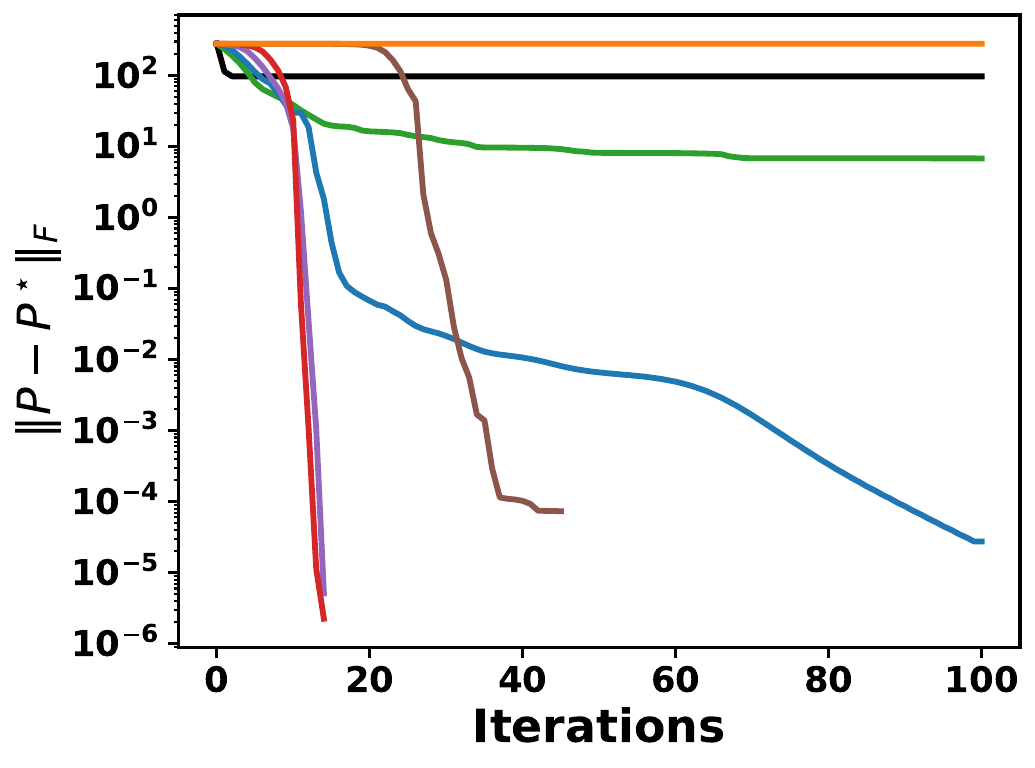}
        \caption{RTR, $\kappa(P^\star)=10^3$}
    \end{subfigure}
    \hfill
    \begin{subfigure}[t]{0.235\textwidth}
        \centering
        \includegraphics[width=\textwidth]{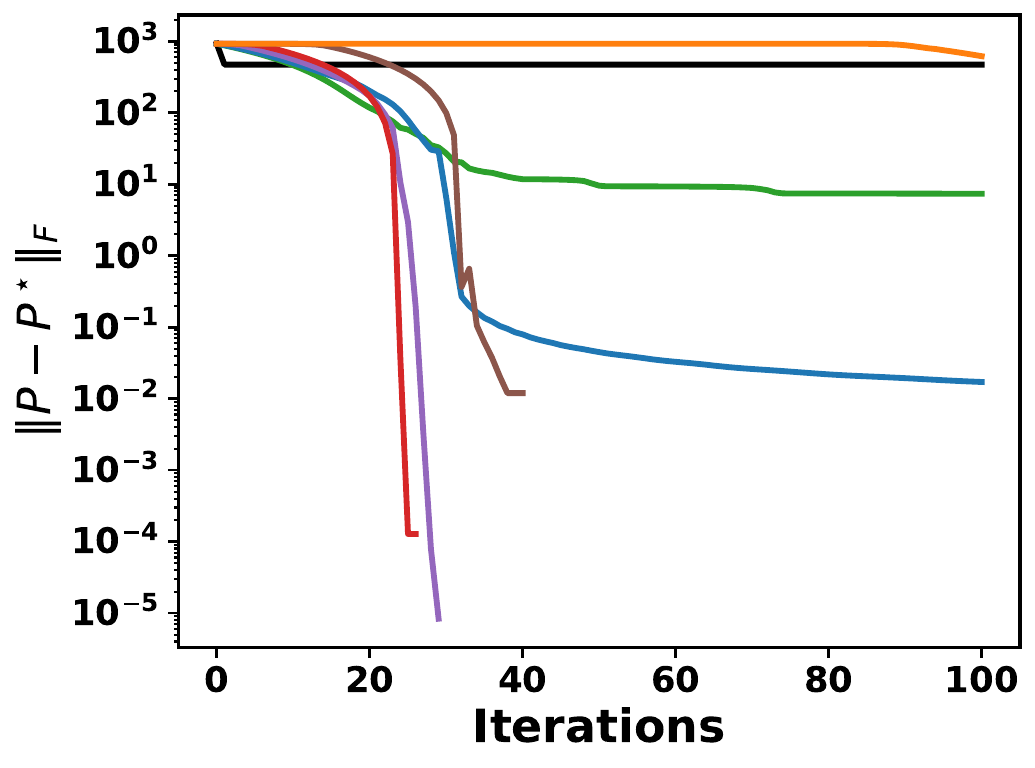}
        \caption{RTR, $\kappa(P^\star)=10^4$}
    \end{subfigure}


        \includegraphics[width=1.0\textwidth]{sylvester_custom_legend_2rows.pdf}


    \caption{Convergence curves for the Sylvester equation on $\SPD(n)$ under different condition numbers of the target solution $P^\star$. The first row shows the RSD results, while the second row shows the corresponding RTR results.}
    \label{fig:sylvester_rsd_rtr_all}
\end{figure*}

\begin{table*}[t]
\centering
\caption{Summary of RSD results for the Sylvester equation.
For each setting, we report the iteration count and runtime.
The best result in each block is underlined.}
\label{tab:ly_rsd_summary}
\resizebox{\textwidth}{!}{
\begin{tabular}{l rr rr rr rr}
\toprule
\multirow{2}{*}{metric}
& \multicolumn{2}{c}{$\mathrm{RSD},\ \kappa(P^\star)=10^{1}$}
& \multicolumn{2}{c}{$\mathrm{RSD},\ \kappa(P^\star)=10^{3}$}
& \multicolumn{2}{c}{$\mathrm{RSD},\ \kappa(P^\star)=10^{5}$}
& \multicolumn{2}{c}{$\mathrm{RSD},\ \kappa(P^\star)=10^{7}$} \\
\cmidrule(lr){2-3}\cmidrule(lr){4-5}\cmidrule(lr){6-7}\cmidrule(lr){8-9}
& \#iter & time (s)
& \#iter & time (s)
& \#iter & time (s)
& \#iter & time (s) \\
\midrule
AI
& 20000 & 26.737
& 20000 & 26.676
& 20000 & 27.654
& 20000 & 27.227 \\

LE ($\alpha=0$)
& 20000 & 28.555
& 20000 & 28.434
& 20000 & 28.398
& 20000 & 28.945 \\

BW ($\alpha=0.5$)
& 20000 & 45.980
& 20000 & 45.470
& 20000 & 45.401
& 20000 & 45.904 \\

$\alpha=0.75$
& 19286 & 44.082
& 20000 & 45.371
& 20000 & 45.068
& 20000 & 45.390 \\

$\alpha=1$
& 10693 & 24.331
& \underline{9741} & \underline{21.979}
& \underline{9945} & \underline{22.238}
& \underline{9713} & \underline{21.841} \\

$\alpha=1.25$
& 6780 & 15.473
& 10636 & 24.594
& 20000 & 45.929
& 20000 & 46.338 \\

$\alpha=1.5$
& \underline{4991} & \underline{11.454}
& 20000 & 47.917
& 20000 & 46.099
& 20000 & 46.454 \\
\bottomrule
\end{tabular}}
\end{table*}

\begin{table*}[t]
\centering
\caption{Summary of RTR results for the Sylvester equation.
For each setting, we report the iteration count and runtime.
The best result in each block is underlined.}
\label{tab:ly_rtr_summary}
\resizebox{\textwidth}{!}{
\begin{tabular}{l rr rr rr rr}
\toprule
\multirow{2}{*}{metric}
& \multicolumn{2}{c}{$\mathrm{RTR},\ \kappa(P^\star)=10^{1}$}
& \multicolumn{2}{c}{$\mathrm{RTR},\ \kappa(P^\star)=10^{2}$}
& \multicolumn{2}{c}{$\mathrm{RTR},\ \kappa(P^\star)=10^{3}$}
& \multicolumn{2}{c}{$\mathrm{RTR},\ \kappa(P^\star)=10^{4}$} \\
\cmidrule(lr){2-3}\cmidrule(lr){4-5}\cmidrule(lr){6-7}\cmidrule(lr){8-9}
& \#iter & time (s)
& \#iter & time (s)
& \#iter & time (s)
& \#iter & time (s) \\
\midrule
AI
& 100 & 40.567
& 100 & 39.913
& 100 & 40.395
& 100 & 40.562 \\

LE ($\alpha=0$)
& 100 & 2.349
& 100 & 2.131
& 100 & 15.151
& 100 & 33.256 \\

BW ($\alpha=0.5$)
& 16 & 1.884
& 22 & 5.134
& 100 & 44.835
& 100 & 34.050 \\

$\alpha=0.75$
& \underline{9} & 0.421
& \underline{10} & 0.680
& \underline{14} & 1.259
& 29 & 7.532 \\

$\alpha=1$
& 14 & \underline{0.266}
& 12 & \underline{0.268}
& \underline{14} & \underline{0.665}
& \underline{26} & \underline{3.188} \\

$\alpha=1.25$
& 36 & 0.444
& 31 & 0.351
& 45 & 3.243
& 40 & 4.427 \\

$\alpha=1.5$
& 100 & 0.935
& 100 & 0.940
& 100 & 1.950
& 100 & 3.911 \\
\bottomrule
\end{tabular}}
\end{table*}

\section{Conclusion}
In this paper, we analyzed the Alpha-Procrustes \((\mathrm{AP})\) geometry for
Riemannian optimization on the \(\mathrm{SPD}\) matrix manifold. Our results
show that, within the AP family, the metric with \(\alpha=1\) is particularly
robust for ill-conditioned optimization problems: its Riemannian Hessian condition number is bounded
independently of the condition number of the underlying SPD matrix, while the
broader \(\mathrm{AP}\) family retains nonnegative sectional curvature. These theoretical findings lead to
improved local convergence guarantees for Riemannian steepest descent and
better conditioning-dependent constants for Riemannian trust-region methods.  We
further established a geodesic convexity transfer principle from the
\(\mathrm{AP}_{1/2}\) \((\mathrm{BW})\) geometry to general
\(\mathrm{AP}_\alpha\) geometries.
Numerical experiments on weighted least squares, trace regression, and the
Sylvester equation confirm that, among the \(\alpha\)-family and the AI
metric, the Riemannian metric with \(\alpha=1\) provides stable and effective performance
for optimization problems involving ill-conditioned SPD matrices.
\section*{Acknowledgment}
This work was supported by JSPS, KAKENHI Grant Number JP25H01112, JP25H01124, 
JP24K15120, JP24H00247, JP26K02871,
Japan and JST, CREST Grant Number JPMJCR22D3, Japan.

\section*{Declarations}

\textbf{Conflict of interest.}
The authors declare that they have no conflict of interest.

\begin{appendices}






\section{Alpha-Procrustes geometry of SPD matrices}
\label{sec: details of alpha procrustes geometry}
Here, we include a complete summary of the Alpha-Procrustes geometry. We refer the reader to \cite{minh2022alpha} for a more detailed discussion.

Fix two Riemannian manifolds $(\mathcal{M},g)$ and $(\mathcal{N},h)$.
Recall that a smooth map
$
\pi:(\mathcal M,g)\to(\mathcal N,h)
$
is called a \emph{smooth submersion} if its differential
\[
D\pi(A):T_A\mathcal M\to T_{\pi(A)}\mathcal N
\]
is surjective for every \(A\in\mathcal M\). Since \(T_A\mathcal M\) is an inner-product space, it admits the orthogonal decomposition
\[
T_A\mathcal M
=
\mathcal V_A\oplus \mathcal H_A
\quad\text{with}\quad
\mathcal{V}_{A}:=\ker(D\pi(A))
\quad\text{and}\quad
\mathcal{H}_{A}:=\bigl(\ker(D\pi(A))\bigr)^\perp,
\]
where \(\mathcal V_A\) and \(\mathcal H_A\) are called the vertical and horizontal subspaces at \(A\), respectively. Because \(D\pi(A)\) is surjective, the restricted map
\[
D\pi(A):\mathcal H_A\to T_{\pi(A)}\mathcal N
\]
is a linear isomorphism. The map \(\pi\) is called a \emph{Riemannian submersion} if, for every \(A\in\mathcal M\), this restricted differential is an isometry, that is,
\[
h_{\pi(A)}\bigl(D\pi(A)[\xi],D\pi(A)[\eta]\bigr)
=
g_A(\xi,\eta)
\qquad
\text{for all }\xi,\eta\in\mathcal H_A.
\]

Fix \(\alpha\in\mathbb R\setminus\{0\}\) and define
\[
\pi_\alpha(A):=(\alpha^2 AA^\top)^{\frac1{2\alpha}}
=\exp\!\left(\frac1{2\alpha}\log(\alpha^2 AA^\top)\right).
\]
Then \(\pi_\alpha\) is a Riemannian submersion from
\(\bigl(\mathrm{GL}(n),\langle\cdot,\cdot\rangle_F\bigr)\) onto
\(\bigl(\SPD(n),g^{(\alpha)}\bigr)\).
For the submersion \(\pi_\alpha\), the ambient manifold \(\mathrm{GL}(n)\) is endowed
with the Frobenius metric. For each \(A_0\in \mathrm{GL}(n)\), the vertical space is
defined by
\begin{equation}
\label{eq:vertical_space_def}
\mathcal V_{A_0}
:=
\ker\!\big(D\pi_\alpha(A_0)\big).
\end{equation}
A direct computation (Proposition 1 of \cite{minh2022alpha}) shows that
\[
\mathcal V_{A_0}
=
\{X\in \mathrm{M}(n):XA_0^\top+A_0X^\top=0\}
=
\Skew(n)\,(A_0^\top)^{-1},
\]
where
$
\Skew(n):=\{S\in\mathbb R^{n\times n}:S^\top=-S\}.
$
Its Frobenius-orthogonal complement is therefore given by
\begin{equation}
\label{eq:horizontal_space_def}
\mathcal H_{A_0}
:=
\mathcal V_{A_0}^\perp
=
\Sym(n)\,A_0.
\end{equation}
Hence one obtains the orthogonal decomposition
\begin{equation}
\label{eq:tangent_decomposition_GL}
T_{A_0}\mathrm{GL}(n)
=
\mathcal V_{A_0}\oplus\mathcal H_{A_0}
=
\Skew(n)\,(A_0^\top)^{-1}\oplus \Sym(n)\,A_0.
\end{equation}

Now fix \(A_0\in \mathrm{GL}(n)\) and a direction \(X\in \mathrm{M}(n)\), and set
\[
P:=\pi_\alpha(A_0)\in\SPD(n),
\;\text{equivalently},\;P^{2\alpha}=\alpha^2A_0A_0^\top.
\]
Then, the differential of \(\pi_\alpha\) at \(A_0\) in the direction \(X\) is given by
\begin{equation}
\label{eq:Dpi_alpha_general}
D\pi_\alpha(A_0)[X]
=
\frac{\alpha}{2}\,
D\exp(\log P)\circ D\log(P^{2\alpha})
\big[XA_0^\top+A_0X^\top\big].
\end{equation}

We now introduce the notion of the horizontal lift.
Let $X \in T_{\pi_\alpha(A_0)}\SPD(n)$ be a tangent vector at
$P=\pi_\alpha(A_0)$. Its \emph{horizontal lift} at $A_0$ is defined as the unique vector
$\widetilde{X}\in \mathcal H_{A_0}$ satisfying
\[
D\pi_\alpha(A_0)[\widetilde{X}] = X.
\]
The notion of horizontal lift plays a fundamental role in the corresponding
Alpha-Procrustes geometry. In particular, it allows one to derive the explicit
expression of the exponential map, as well as the formulas for computing the
associated Riemannian gradient and Riemannian Hessian.

\subsection{Riemannian gradient via horizontal lift}
\label{subsec:riemannian_gradient_horizontal_lift}

The Riemannian gradient is written by the horizontal lift.
For any twice continuously differentiable function
$
f:\SPD(n)\to\mathbb{R}
$
and its lifted function
$
\widetilde f:=f\circ\pi_\alpha:\mathrm{GL}(n)\to\mathbb{R}
$, 
the Riemannian gradient of \(f\) on \(\SPD(n)\) satisfies the identity
\[
\grad^{(\alpha)} f(P)
=
D\pi_\alpha(A_0)\big[\nabla \widetilde f(A_0)\big],
\qquad
P=\pi_\alpha(A_0),
\]
for any \(A_0\in \mathrm{GL}(n)\); see \cite[Proposition~9.39]{boumal2023introduction}. Here \(\nabla \widetilde f(A_0)\) denotes the Euclidean gradient of \(\widetilde f\) on \(\mathrm{GL}(n)\) with respect to the Frobenius inner product. Moreover, \(\nabla \widetilde f(A_0)\in\mathcal H_{A_0}\) is precisely the horizontal lift of \(\grad^{(\alpha)} f(P)\).

\subsection{Exponential map in the Alpha-Procrustes geometry}
\label{subsec:expmap_alpha_procrustes}

The exponential map is written by the O'Neill geodesic projection principle.
\begin{lemma}[O'Neill's geodesic projection principle {\normalfont\cite[Proposition~2.109]{gallot1990riemannian}}]
\label{lem:oneill_geodesic_projection}
Let
$\pi:(\mathcal M,g)\to(\mathcal N,h)$
be a Riemannian submersion, and let $\mathcal H_A\subset T_A\mathcal M$ denote the corresponding horizontal subspace at $A\in\mathcal M$.
If $A:I\to\mathcal M$ is a horizontal geodesic, that is,
\[
\nabla^{\mathcal M}_{\dot A(t)}\dot A(t)=0,
\qquad
\dot A(t)\in \mathcal H_{A(t)},
\qquad t\in I,
\]
then the projected curve
$\gamma(t):=\pi(A(t))$
is a geodesic in $(\mathcal N,h)$.
\end{lemma}

We now apply Lemma~\ref{lem:oneill_geodesic_projection} to the Alpha-Procrustes geometry.
\begin{theorem}[Exponential map for \(\alpha\neq 0\)]
\label{thm:Exp_alpha_general}
Let \(\alpha\in\mathbb R\setminus\{0\}\), \(P\in\SPD(n)\), and \(X\in \Sym(n)\).
Define \(Y\in\Sym(n)\) by
\begin{equation}
\label{eq:def_Y_from_X}
(D\exp)(\log P)\circ (D\log)(P^{2\alpha})
\big[YP^{2\alpha}+P^{2\alpha}Y\big]
=
2\alpha\,X.
\end{equation}
Equivalently,
$
Y=\mathcal L_{P,\alpha}(2\alpha\,X).
$
Then the exponential map at \(P\) is given by
\begin{equation}
\label{eq:geodesic_IVP_unified}
\Exp_P^{(\alpha)}(tX)
=
\Big((I+tY)\,P^{2\alpha}\,(I+tY)\Big)^{\frac1{2\alpha}}.
\end{equation}
In particular, we have
\begin{equation}
\label{eq:expmap_general_unified}
\Exp_P^{(\alpha)}(X)
=
\Big((I+Y)\,P^{2\alpha}\,(I+Y)\Big)^{\frac1{2\alpha}},
\qquad
Y=\mathcal L_{P,\alpha}(2\alpha\,X).
\end{equation}
\end{theorem}

\begin{proof}
Choose $A_0\in \mathrm{GL}(n)$ such that $\pi_\alpha(A_0)=P$. Since
\[
\pi_\alpha:\bigl(\mathrm{GL}(n),\langle\cdot,\cdot\rangle_F\bigr)\to \bigl(\SPD(n),g^{(\alpha)}\bigr)
\]
is a Riemannian submersion with horizontal space
\[
\mathcal H_A=\Sym(n)\,A,
\]
it suffices to construct a horizontal geodesic in $\mathrm{GL}(n)$ projecting to the desired curve.

Let $Y\in\Sym(n)$ and define
\[
A(t):=(I+tY)A_0
\]
for $t$ in a sufficiently small interval around $0$ such that $I+tY$ is invertible. Since $\mathrm{GL}(n)$ is an open subset of $\mathrm{M}(n)$ endowed with the Frobenius metric, $A(t)$ is a geodesic. Moreover,
\[
\dot A(t)=YA_0=Y(I+tY)^{-1}A(t).
\]
Because $Y\in\Sym(n)$ and $Y$ commutes with $(I+tY)^{-1}$, the matrix $Y(I+tY)^{-1}$ is symmetric. Hence
\[
\dot A(t)\in \Sym(n)\,A(t)=\mathcal H_{A(t)},
\]
so $A(t)$ forms a horizontal geodesic. Therefore, by Lemma~\ref{lem:oneill_geodesic_projection}, the projected curve
\[
\gamma(t):=\pi_\alpha(A(t))
\]
is a geodesic in $\SPD(n)$.

We now write $\gamma(t)$ explicitly.
Using $P^{2\alpha}=\alpha^2A_0A_0^\top$, we obtain
\[
\gamma(t)^{2\alpha}
=
\alpha^2A(t)A(t)^\top
=
\alpha^2(I+tY)A_0A_0^\top(I+tY)^\top
=
(I+tY)P^{2\alpha}(I+tY),
\]
and hence we obtain
\[
\gamma(t)=\Big((I+tY)\,P^{2\alpha}\,(I+tY)\Big)^{\frac1{2\alpha}}.
\]

It remains to match the initial tangent vector. Set
\[
S(t):=(I+tY)P^{2\alpha}(I+tY),
\qquad
g(Z):=Z^{1/(2\alpha)}=\exp\!\Big(\frac{1}{2\alpha}\log Z\Big).
\]
Then we have $\gamma(t)=g(S(t))$,
\[
S(0)=P^{2\alpha},
\qquad
\dot S(0)=YP^{2\alpha}+P^{2\alpha}Y,
\]
and therefore, by the chain rule, we get
\[
\dot\gamma(0)
=
Dg(P^{2\alpha})[\dot S(0)]
=
\frac{1}{2\alpha}\,
D\exp(\log P)\circ D\log(P^{2\alpha})
\big[YP^{2\alpha}+P^{2\alpha}Y\big].
\]
Thus $\dot\gamma(0)=X$ if and only if
\[
D\exp(\log P)\circ D\log(P^{2\alpha})
\big[YP^{2\alpha}+P^{2\alpha}Y\big]
=
2\alpha\,X,
\]
that is,
\[
Y=\mathcal L_{P,\alpha}(2\alpha\,X).
\]
Substituting this into the above expression for $\gamma(t)$ yields
\[
\Exp_P^{(\alpha)}(tX)
=
\Big((I+tY)\,P^{2\alpha}\,(I+tY)\Big)^{\frac1{2\alpha}},
\]
and in particular
\[
\Exp_P^{(\alpha)}(X)
=
\Big((I+Y)\,P^{2\alpha}\,(I+Y)\Big)^{\frac1{2\alpha}}.
\]
\end{proof}

\subsection{Affine connection and Riemannian Hessian induced by the Riemannian submersion}
\label{hession_submersion}


Let $\overline{\nabla}$ denote the Levi--Civita connection of
$(\mathrm{GL}(n),\langle\cdot,\cdot\rangle_F)$.
Since $\mathrm{GL}(n)$ is an open subset of the Euclidean space
$(\mathrm{M}(n),\langle\cdot,\cdot\rangle_F)$, the Levi--Civita connection
coincides with the flat connection.
More precisely, let \(X,Y\) be smooth vector fields on \(\mathrm{GL}(n)\). Then the connection is given by
\[
(\overline{\nabla}_X Y)_{A_0}
=
DY(A_0)[X(A_0)],
\qquad
A_0\in \mathrm{GL}(n),
\]
where \(DY(A_0):T_{A_0}\mathrm{GL}(n)\to \mathrm{M}(n)\) is the differential of the map
\(Y:\mathrm{GL}(n)\to \mathrm{M}(n)\) at \(A_0\).

For $X,Y\in T_P\SPD(n)$, let $\widetilde{X},\widetilde{Y}\in\mathcal H_{A_0}$ denote their
horizontal lifts at $A_0$, i.e.,
\[
D\pi_\alpha(A_0)[\widetilde{X}] = X,
\qquad
D\pi_\alpha(A_0)[\widetilde{Y}] = Y.
\]
In this representation, the horizontal lifts are given by
\[
\widetilde{X}=S_X A_0,
\qquad
\widetilde{Y}=S_Y A_0,
\qquad
S_X,S_Y\in\Sym(n).
\]
By \eqref{eq:horizontal_space_def} and \eqref{eq:tangent_decomposition_GL}, the orthogonal projection
\[
P_{\mathcal H_{A_0}}:T_{A_0}\mathrm{GL}(n)\to \mathcal H_{A_0}
\]
is well defined. For any \(Z\in T_{A_0}\mathrm{GL}(n)\simeq \mathrm{M}(n)\), since \(\mathcal H_{A_0}=\Sym(n)A_0\), there exists a unique matrix \(S_Z\in\Sym(n)\) such that
\begin{equation}
\label{eq:PH_explicit}
P_{\mathcal H_{A_0}}(Z)=S_ZA_0.
\end{equation}
where \(S_Z\in\Sym(n)\) is the unique solution of the Lyapunov equation
\begin{equation}
\label{eq:PH_sylvester}
S_ZM+MS_Z
=
ZA_0^\top+A_0Z^\top,
\qquad
M:=A_0A_0^\top\in\SPD(n).
\end{equation}
Since \(M\in\SPD(n)\), the operator \(S\mapsto SM+MS\) is invertible on \(\Sym(n)\), so \(S_Z\) is uniquely determined.

Let \(U,V\) be smooth vector fields on \(\SPD(n)\), and let \(\widetilde U,\widetilde V\) be their horizontal lifts to \(\mathrm{GL}(n)\). Then the Levi--Civita connection on \((\SPD(n),g^{(\alpha)})\) is obtained by projecting the lifted ambient connection onto the horizontal space; see \cite[Proposition~5.3.4]{absil2008optimization}.
\begin{equation}
\label{eq:alpha_connection_projection}
\bigl(\nabla^{(\alpha)}_{U}V\bigr)_P
=
D\pi_\alpha(A_0)
\Big[
P_{\mathcal H_{A_0}}
\bigl(
(\overline{\nabla}_{\widetilde{U}}\widetilde{V})_{A_0}
\bigr)
\Big]
=
D\pi_\alpha(A_0)
\Big[
P_{\mathcal H_{A_0}}
\bigl(
D\widetilde{V}(A_0)[\widetilde{U}(A_0)]
\bigr)
\Big],
\end{equation}
where \(P=\pi_\alpha(A_0)\).

Let \(\widetilde X\in\mathcal H_{A_0}\) be the horizontal lift of \(X\in T_P\SPD(n)\). By the discussion in \S\ref{subsec:riemannian_gradient_horizontal_lift}, \(\nabla \widetilde f(A_0)\in\mathcal H_{A_0}\) is the horizontal lift of \(\grad^{(\alpha)}f(P)\). Therefore, using \eqref{eq:alpha_connection_projection} and the flatness of \(\overline{\nabla}\), we obtain
\begin{equation}
\label{eq:alpha_hessian_explicit}
\mathrm{Hess}^{(\alpha)} f(P)[X]
=
D\pi_\alpha(A_0)
\Big[
P_{\mathcal H_{A_0}}\big(
D(\nabla \widetilde f)(A_0)[\widetilde X]
\big)
\Big].
\end{equation}

\section{Proof of Proposition~\ref{prop:ap_nonnegative_curvature}}

\label{sec:proof_ap_nonnegative_curvature}
\begin{proof}
We use the quotient representation of the AP geometry. Let
\(\mathrm{GL}(n)\) be endowed with the Frobenius metric, which is flat. For
\(G\in\mathrm{GL}(n)\), define
\[
    \pi_\alpha:\mathrm{GL}(n)\to\SPD,
    \qquad
    \pi_\alpha(G):=(GG^\top)^{1/(2\alpha)}.
\]
Equivalently, for every \(G\in\mathrm{GL}(n)\),
\[
    \Phi_\alpha\bigl(\pi_\alpha(G)\bigr)
    =
    \bigl(\pi_\alpha(G)\bigr)^{2\alpha}
    =
    GG^\top,
\]
where \(\Phi_\alpha(P)=P^{2\alpha}\). Thus \(\pi_\alpha\) is the AP analogue of the standard Procrustes quotient map for the
\(\mathrm{BW}\) geometry.

By the quotient construction of the \(\AP_\alpha\) metric
\cite{minh2022alpha}, the map \(\pi_\alpha\) is a Riemannian submersion from
the flat manifold
$
    \bigl(\mathrm{GL}(n),\langle\cdot,\cdot\rangle_F\bigr)
$
onto \((\SPD,g^{\AP})\).
For this submersion, the vertical and horizontal spaces are defined by
\[
    \mathcal V_G:=\ker(D\pi_\alpha(G)),
    \qquad
    \mathcal H_G:=\mathcal V_G^\perp,
\]
and, as summarized in Appendix~\ref{sec: details of alpha procrustes geometry},
they satisfy
\[
    T_G\mathrm{GL}(n)
    =
    \mathcal V_G\oplus\mathcal H_G,
    \qquad
    \mathcal H_G=\Sym(n)\,G;
\]
Moreover, for any
\(X\in T_P\SPD\), its horizontal lift at \(G\in\pi_\alpha^{-1}(P)\) is the
unique vector \(\widetilde X\in\mathcal H_G\) satisfying
$
    D\pi_\alpha(G)[\widetilde X]=X;
$
please also refer to Appendix~\ref{sec: details of alpha procrustes geometry} in detail.

Let \(\sigma\subset T_P\SPD\) be a two-dimensional tangent plane, and let
\(\widetilde{\sigma}\subset\mathcal H_G\) be its horizontal lift at some
\(G\in\pi_\alpha^{-1}(P)\). Choose linearly independent horizontal vectors
\(\widetilde U,\widetilde V\in\widetilde{\sigma}\). By O'Neill's curvature
formula for Riemannian submersions \cite[Section~2.6]{gallot1990riemannian},
we have
\[
    K^{\AP_\alpha}(P;\sigma)
    =
    K^{\mathrm{GL}(n)}(G;\widetilde{\sigma})
    +
    \frac{3}{4}
    \frac{
        \|[\widetilde U,\widetilde V]^{\mathcal V_G}\|_F^2
    }{
        \|\widetilde U\|_F^2\|\widetilde V\|_F^2
        -
        \langle \widetilde U,\widetilde V\rangle_F^2
    },
\]
where \([\widetilde U,\widetilde V]^{\mathcal V_G}\) denotes the vertical
component of the Lie bracket with respect to the orthogonal decomposition
\(T_G\mathrm{GL}(n)=\mathcal V_G\oplus\mathcal H_G\).

Since \(\mathrm{GL}(n)\) is an open subset of the Euclidean space
\(\mathrm{M}(n)\) endowed with the Frobenius metric, it is flat. Hence
\[
    K^{\mathrm{GL}(n)}(G;\widetilde{\sigma})=0.
\]
The second term in O'Neill's formula is nonnegative. Therefore,
\[
    K^{\AP_\alpha}(P;\sigma)\ge 0
\]
for every \(P\in\SPD\) and every two-dimensional tangent plane
\(\sigma\subset T_P\SPD\). Hence the \(\AP\) geometry has
nonnegative sectional curvature.
\end{proof}




\end{appendices}


\bibliography{sn-bibliography}

\end{document}